\documentclass[12pt,twoside,reqno]{amsart}
\usepackage{mathptmx, amsmath, amssymb, amsfonts, amsthm, enumerate, mathrsfs}
\usepackage{xcolor}
\usepackage[colorlinks=true,citecolor=blue]{hyperref}

\hypersetup{
  pdftitle={Directional Subdifferentials of the Value Function in Asplund Spaces},
  pdfauthor={Weihao Mao and Jane J. Ye}
}

\newtheorem{theorem}{Theorem}[section]

\newtheorem{proposition}[theorem]{Proposition}
\newtheorem{corollary}[theorem]{Corollary}

\theoremstyle{definition}
\newtheorem{definition}[theorem]{Definition}

\newtheorem{example}[theorem]{Example}

\newtheorem{remark}[theorem]{Remark}
\numberwithin{equation}{section}

\begin{document}
\title[Directional Subdifferentials of the Value Function] 
{Directional Subdifferentials of the Value Function in Asplund Spaces}

\author[W. Mao]{Weihao Mao}
\address{Department of Mathematics and Statistics, University of Victoria, Victoria, BC V8W 2Y2, Canada}
\email{mwhaea123456@gmail.com}

\author[J. J. Ye]{Jane J. Ye}
\address{Department of Mathematics and Statistics, University of Victoria, Victoria, BC V8W 2Y2, Canada}
\email{janeye@uvic.ca}
\thanks{Corresponding author: Jane J. Ye.}


\begin{abstract}
Directional subdifferentials of the value function provide a quantitative measure of optimal value response to perturbations. While existing results are largely limited to finite-dimensional settings, this paper develops a comprehensive variational framework in Asplund spaces. We establish essential directional calculus rules, extending directional nonsmooth analysis to infinite dimensions. 
To address the lack of compactness of bounded sets in infinite-dimensional spaces, we introduce a new directional condition, under which we derive upper estimates for directional limiting and singular subdifferentials of the value function. These results provide a refined analytical foundation for sensitivity analysis in infinite-dimensional hierarchical systems. 
\end{abstract}

\maketitle

\section{Introduction}

In this paper, we consider a parametric optimization problem of the form 
\begin{equation}\label{eq: basic model of lower level}
    \begin{split}
        &\min_{y \in Y}~f(x,y)\quad \text{s.t.}\quad P(x,y)\in \Lambda,
    \end{split}
\end{equation}
where $x\in X$, $f:X\times Y\to \mathbb{R}$, $P:X\times Y\to Z$ are locally Lipschitzian functions, and $\Lambda\subset Z$ is a closed set, with $X,Y,Z$ being Asplund spaces (e.g., $L^p$ spaces with $1<p<\infty$, Hilbert spaces), unless otherwise specified.   Problem~\eqref{eq: basic model of lower level} provides a general optimization framework that encompasses a broad class of models, including general parametric optimization problems; see, e.g., \cite{bonnans2013perturbation} or a mathematical program with geometric constraints; see, e.g., \cite{guo2013mathematical}.  

In practice, it is crucial to understand how the optimal value of \eqref{eq: basic model of lower level} varies with respect to perturbations of the parameter $x$. 
To this end, we define the feasible set mapping $\Gamma(x):=\{y \in Y \mid P(x,y)\in \Lambda \}$, 
and the associated value function $V(x):=\inf\{f(x,y)\mid P(x,y)\in\Lambda\}$,  
which characterizes the dependence of the optimal value of \eqref{eq: basic model of lower level} on the parameter $x$. 
We also define the solution mapping $S(x):=\{y\in \Gamma(x)\mid f(x,y)=V(x)\}$. 

The differentiability properties of $V$ are of central importance in sensitivity analysis and numerical optimization. 
For instance, one of the most powerful approaches to solving bilevel optimization problems is the \emph{value function approach}, which relies heavily on the analytical differentiability properties of $V$; see \cite{Ye1995,Ye2010}.  

In many situations, it is better to analyze differentials of functions only in certain directions rather than in all possible ones. Motivated by this observation, the notion of the \emph{directional subdifferential} was introduced independently by Ginchev and Mordukhovich~\cite{ginchev2011directionally} and by Gfrerer~\cite{gfrerer2013directional}. Subsequently, a comprehensive framework of \emph{directional variational analysis} has been developed (see, e.g., \cite{Ginchev2012directional,gfrerer2013directional,long2017calculus,long2017calculusb}), which provides refined tools for characterizing local sensitivity and establishes a rich variational calculus with numerous successful applications to optimization and equilibrium problems. These developments show that directional optimality conditions are typically sharper than their nondirectional counterparts, while directional constraint qualifications are often weaker \cite{gfrerer2013directional,bai2022directional}.

Despite advances, most existing results are confined to finite-dimensional settings, where the framework of directional variational analysis is well established for directional subdifferentials of the value function $V$ (see, e.g., \cite{benko2019calculus,bai2022directional,Bai2023Directional}). 
Only a few studies, such as \cite{gfrerer2013directional,long2017calculus,long2017calculusb}, have extended these developments to infinite-dimensional Banach spaces. 
However, the resulting formulas for the directional subdifferentials of $V$, which often involve coderivative terms, are generally implicit and difficult to compute~\cite{long2017calculus}. 
This challenge stems from the lack of local compactness and the complex structure of normal cones in infinite-dimensional spaces \cite{gfrerer2013directional}.  

In this paper, we derive upper estimates for the directional limiting and directional singular subdifferentials of the value function associated with Problem~\eqref{eq: basic model of lower level} under certain conditions. These estimates are established under mild assumptions, specifically directional restricted inf-compactness and directional metric subregularity. These results provide a refined analytical foundation for sensitivity analysis and the derivation of optimality conditions in infinite-dimensional hierarchical systems.

Specifically, we develop a variational framework that extends directional subdifferential analysis from finite-dimensional settings to infinite-dimensional Asplund spaces. This framework includes comprehensive calculus rules, such as those for preimage normal cones (Theorem~\ref{thm: directional normal cones of preimage sets}), chain rule (Theorem~\ref{thm: chain rule}), and sum rules (Theorem~\ref{thm: sum rule}).

To overcome the challenges arising from the lack of compactness in infinite-dimensional spaces, our analysis initially considers strong directional inner semicompactness. To accommodate a broader class of problems, we subsequently propose a weaker assumption for \eqref{eq: basic model of lower level}, termed the directional $V$-flatness condition (Definition~\ref{defn: gamma V flatness}). We demonstrate that this condition can be verified under certain regularity properties of $f$ and $\Gamma$ (as shown in Proposition~\ref{prop: first_order_V_flatness_verification}). By relying on this weaker flatness condition, Theorem~\ref{thm: directional subdifferential of V} derives directional subdifferential estimates of \(V\). Notably, when the solution map $S$ satisfies the stronger inner semicompactness condition, our theorem refines the results established in \cite[Theorems~5.10 and~5.11]{long2017calculus}. 

The remainder of the paper is organized as follows. 
Section~2 introduces the necessary background and preliminaries, establishing the basic framework of directional variational analysis. 
Section~3 presents our main results on the directional limiting and singular subdifferentials of the value function. 

\section*{Notation}
Throughout this paper, $X$, $Y$, and $Z$ denote Banach spaces, with their respective topological dual spaces represented by $X^*$, $Y^*$, and $Z^*$. The closed unit balls in these spaces and their duals are denoted by $B_X, B_Y, B_Z$ and $B_{X^*}, B_{Y^*}, B_{Z^*}$, respectively. The operator $\nabla$ specifically represents the Fr\'echet derivative. For nonsmooth functions, the symbols $\hat{\partial}$ and $\partial$ denote the Fr\'echet and limiting (Mordukhovich) subdifferentials, respectively. Unless otherwise specified, limits and convergence are understood to be with respect to the norm topology. The notations $w$ and $w^*$ are used to signify the weak and weak-star topologies or limits, respectively. 

\section{Preliminaries}
In this section, we recall several fundamental notions and properties from variational analysis that will be used throughout the paper. 
We also introduce their directional counterparts and establish some of their key properties. 
Several of these results are infinite-dimensional generalizations of their finite-dimensional counterparts in \cite{gfrerer2013directional,Bai2023Directional}.  
For additional background on variational analysis in Asplund and general Banach spaces, we refer the reader to \cite{Mordukhovich2006variational}.  

\subsection{Preliminaries on normal cones and coderivatives}\label{subsec for preliminary on variational analysis}

In this paper, we primarily work with the Fr\'echet and Mordukhovich (limiting) normal cones, as well as the coderivatives and their directional variants. 
We first recall some fundamentals that will be used throughout the paper.

Let $X$ be a Banach space, and let $\Omega\subset X$ be a nonempty subset. For $\bar x\in\Omega$, the \emph{contingent cone} (the Bouligand tangent cone) to $\Omega$ at $\bar x$ is defined by
\begin{equation*}
    T(\bar x;\Omega)
    :=\{u\in X \mid ~ \exists\, x_k\in \Omega,~ t_k\downarrow 0 \text{ such that } (x_k-\bar x)/t_k \to u \}.
\end{equation*}
It is well known that $T(\bar x;\Omega)$ is closed but not necessarily convex. For a fixed $\varepsilon \ge 0$, the set of \emph{$\varepsilon$-normals} to $\Omega$ at $\bar x \in \Omega$ is defined as
\begin{equation*}
    \hat{N}_{\varepsilon}(\bar x;\Omega)
    := \big\{x^* \in X^* \mid 
    \limsup_{x \xrightarrow{\Omega} \bar x}
    \frac{\langle x^*,\, x - \bar x \rangle}{\|x - \bar x\|}
    \le \varepsilon \big\},
\end{equation*}
where $x \xrightarrow{\Omega} \bar x$ means that $x \to \bar x$ with $x \in \Omega$. 
Clearly, if $\Omega = \{\bar x\}$, then $\hat{N}_{\varepsilon}(\bar x;\Omega) = X^*$; 
if $\bar x \notin \Omega$, we set $\hat{N}_{\varepsilon}(\bar x;\Omega) = \emptyset$ for any $\varepsilon \ge 0$. 
When $\varepsilon = 0$, the set $\hat{N}(\bar x;\Omega) := \hat{N}_0(\bar x;\Omega)$ is referred to as the \emph{Fr\'echet} (or \emph{prenormal}) cone.

The \emph{Mordukhovich} (or \emph{limiting}) normal cone to $\Omega$ at $\bar x$ is defined as
\begin{equation*}
    N(\bar x;\Omega)
    :=\big\{x^*\in X^*\mid 
    \exists\, \varepsilon_k\downarrow 0,~ x_k\to \bar x,~ x_k^*\overset{w^*}{\to} x^*,~
    x_k^*\in \hat{N}_{\varepsilon_k}(x_k;\Omega)\big\}.
\end{equation*}

When $\Omega$ is convex, both $\hat{N}(\bar x;\Omega)$ and $N(\bar x;\Omega)$ coincide with the classical normal cone in convex analysis. 
Let $\bar u\in X$ be a given direction vector. 
The \emph{directional Mordukhovich normal cone} to $\Omega$ at $\bar x\in \Omega$ in the direction $\bar u$ is defined by
\begin{equation*}
    N(\bar x;\Omega;\bar u)
    :=\limsup_{\varepsilon,t\downarrow 0,~u\to \bar u}
    \hat{N}_{\varepsilon}(\bar x+tu;\Omega),
\end{equation*}
that is, $x^*\in N(\bar x;\Omega;\bar u)$ if there exist sequences 
$\varepsilon_k\downarrow 0$, $t_k\downarrow 0$, $u_k\to \bar u$, and $x_k^*\overset{w^*}{\to}x^*$ such that 
$x_k^*\in \hat{N}_{\varepsilon_k}(\bar x+t_ku_k;\Omega)$ for all $k$. 
When $X$ is an Asplund space and $\Omega$ is closed around $\bar x$, 
the $\varepsilon$-dependence can be omitted (see \cite[Theorem~2.34]{Mordukhovich2006variational}), and hence
\begin{equation*}
    N(\bar x;\Omega;\bar u)
    =\limsup_{t\downarrow 0,~u\to \bar u}
    \hat{N}(\bar x+tu;\Omega).
\end{equation*}
By definition, it follows that $N(\bar{x}; \Omega; \bar{u}) = \emptyset$ whenever $\bar{u} \notin T(\bar{x}; \Omega)$. 
Besides, the following relations hold:
\[
N(\bar x;\Omega;\bar u)=N(\bar x;\Omega;\lambda\bar u),~\forall \lambda>0;~
N(\bar x;\Omega;0)=N(\bar x;\Omega);~
N(\bar x;\Omega;\bar u)\subset N(\bar x;\Omega).
\]

Moreover, directional normal cones satisfy the following product-type rule \cite[(3.14) and Theorem~3.7]{long2017calculus}.

\begin{proposition}\label{prop of directional normal cones of product of separated sets}
Let $\Omega_i\subset X_i$ $(i=1,2)$ be nonempty subsets, and let $\bar x=(\bar x_1,\bar x_2)\in \Omega:=\Omega_1\times \Omega_2$. 
Then, for any $\bar u=(\bar u_1,\bar u_2)\in X_1\times X_2$, one has
\begin{equation}
    N(\bar x;\Omega;\bar u)
    \subset 
    N(\bar x_1;\Omega_1;\bar u_1)\times N(\bar x_2;\Omega_2;\bar u_2).
\end{equation}
\end{proposition}

For a set-valued mapping $F: X \rightrightarrows Y$, the \emph{normal coderivative} of $F$ at $(\bar{x}, \bar{y}) \in \text{gph} F:=\{(x,y)\in X\times Y \mid ~y\in F(x)\}$ is the set-valued mapping $D_N^* F(\bar{x}, \bar{y}): Y^* \rightrightarrows X^*$: 
\[
D_N^*F(\bar x, \bar y)(y^*)
:= \big\{
x^* \in X^* ~\big|~ 
(x^*, -y^*) \in N\big((\bar x, \bar y); \operatorname{gph} F\big)
\big\}, ~\forall\, y^* \in Y^*.
\]
The directional normal coderivative of $F$ at $(\bar x, \bar y)$ in the direction $(\bar u, \bar v)$ is defined by:  
\[
D_N^*F\big((\bar x, \bar y); (\bar u, \bar v)\big)(y^*)
:= \big\{
x^* \in X^* ~\big|~ 
(x^*, -y^*) \in N\big((\bar x, \bar y); \operatorname{gph} F; (\bar u, \bar v)\big)
\big\}, ~\forall\, y^* \in Y^*.
\]
When $F$ is a function $\varphi$, we simplify the notation by omitting $\bar y = \varphi(\bar x)$, writing it as $D_N^* \varphi(\bar x; (\bar u, \bar v))$. Furthermore, if $\varphi$ is Hadamard differentiable (in the following definition), we also omit $\bar v$ in the notation $D_N^* \varphi(\bar x; \bar u)$, as we shall only consider the unique direction $\bar v = \varphi'_H(\bar x; \bar u)$. 

\begin{definition}\label{defn of hadamard derivative} 
Let $\varphi: X \to Y$ be a mapping between Banach spaces. The \emph{Hadamard directional derivative} of $\varphi$ at $\bar x$ in the direction $\bar u$ is defined by:
\begin{equation}
    \varphi'_H(\bar x; \bar u)
    := 
    \lim_{t\downarrow 0, u \to \bar u} 
    \frac{\varphi(\bar x + t u) - \varphi(\bar x)}{t}. 
\end{equation}
The \textit{graphical derivative} of $\varphi$ at $\bar x$ in the direction $\bar u$ is defined as:
\[
D\varphi(\bar x)(\bar u) := \left\{ \xi \in Y \;\middle|\; \exists t_k \downarrow 0, \, u_k \to \bar u \text{ such that } \lim_{k \to \infty} \frac{\varphi(\bar x + t_k u_k) - \varphi(\bar x)}{t_k} = \xi \right\}.
\]
If $\varphi$ is Hadamard directional differentiable at $\bar x$ in the direction $\bar u$, then $D\varphi(\bar x)(\bar u)=\{ \varphi'_H(\bar x; \bar u)\}$.  
For $Y=\mathbb{R}$, the lower / upper Dini directional derivatives of $\varphi$ at $\bar{x}$ in the direction $\bar u$ are defined respectively as:
    \begin{align*}
        \varphi'_{-}(\bar{x}; \bar u) := \liminf_{t \downarrow 0} \frac{\varphi(\bar{x} + t \bar u) - \varphi(\bar{x})}{t}, ~~
        \varphi'_{+}(\bar{x}; \bar u) := \limsup_{t \downarrow 0} \frac{\varphi(\bar{x} + t \bar u) - \varphi(\bar{x})}{t}.
    \end{align*}
\end{definition}

\subsection{Subdifferentials and directional subdifferentials}

In this subsection, we present the fundamental concepts of directional subdifferentials. 
We begin by introducing the notion of a \emph{directional neighborhood}, to describe properties along specified directions.

\begin{definition}\label{defn:directional neighborhood}
Let \( X \) be a Banach space, and \( d \in X \) be a given direction.  
For any \( \varepsilon > 0 \), \( \delta > 0 \), the \emph{directional neighborhood} of \( \bar{x} \) in direction \( d \) is: 
\[
\mathcal{V}_{\varepsilon,\delta}(\bar{x}; d) := \bar x+\{z \in \varepsilon B_X \mid \|~ \|d\|z - \|z\|d ~\| \leq \delta \|z\| \|d\| \}. 
\]
We use $\mathcal{V}(\bar x;d)$ to denote an arbitrary directional neighborhood of $\bar x$ in the direction $d$. 
\end{definition}

Let $\varphi: X \to \overline{\mathbb{R}}$ be an extended-real-valued function and $\bar{x} \in \operatorname{dom} \varphi:=\{x\in X \mid \varphi(x)<\infty\}$. The epigraph of $\varphi$ is defined as
\begin{equation*}
    \begin{split}
        \operatorname{epi} \varphi:=\{(x,r)\in X\times \mathbb{R}\mid ~\varphi(x)\leq r, ~x\in \operatorname{dom} \varphi\}.
    \end{split}
\end{equation*}
The \emph{$\varepsilon$-subdifferential} of $\varphi$ at $\bar{x}$ is defined by
\[
\widehat{\partial}_{\varepsilon}\varphi(\bar{x})
:= \left\{ x^* \in X^* \;\middle|\; \liminf_{x \to \bar{x}} \frac{\varphi(x) - \varphi(\bar{x}) - \langle x^*, x - \bar{x} \rangle}{\|x - \bar{x}\|} \geq -\varepsilon \right\}.
\]
When $\varepsilon = 0$, this set is called the \emph{Fr\'echet subdifferential}, denoted by $\widehat{\partial}\varphi(\bar{x})$. 
The \emph{Mordukhovich (limiting) subdifferential} $\partial \varphi(\bar{x})$ and the \emph{Mordukhovich singular subdifferential} $\partial^{\infty}\varphi(\bar{x})$ of $\varphi$ at $\bar{x}$ are defined, respectively, by
\begin{align*}
    \partial \varphi(\bar{x}) &:= \big\{ x^* \in X^* \mid (x^*, -1) \in N\big((\bar{x}, \varphi(\bar{x})); \operatorname{epi} \varphi\big) \big\}, \\
    \partial^{\infty}\varphi(\bar{x}) &:= \big\{ x^* \in X^* \mid (x^*, 0) \in N\big((\bar{x}, \varphi(\bar{x})); \operatorname{epi} \varphi\big) \big\}.
\end{align*}

The \emph{directional Mordukhovich subdifferential} $\partial \varphi(\bar{x}; \bar{u})$ of $\varphi$ at $\bar{x} \in \operatorname{dom}\varphi$ in direction $\bar{u} \in X$ is defined by
\begin{equation*}
\begin{aligned}
\partial \varphi(\bar{x}; \bar{u})
&:= 
\big\{
x^* \in X^* ~\big|~
\exists\, \varepsilon_k \downarrow 0,~ t_k \downarrow 0,~ u_k \to \bar{u},~
x_k^* \overset{w^*}{\to} x^*,\\
&\quad\quad\quad\quad\quad\quad\text{with } \varphi(\bar{x} + t_k u_k) \to \varphi(\bar{x}) 
\text{ and } x_k^* \in \widehat{\partial}_{\varepsilon_k}\varphi(\bar{x} + t_k u_k)
\big\}. 
\end{aligned}
\end{equation*}

\begin{remark}\label{rem:defn_of_other_directional_subdifferential}
A refined directional subdifferential was defined in \cite[(2)]{benko2019calculus} by taking the direction of the function value into account:
\begin{equation*}
\partial \varphi(\bar{x}; (u, \mu)) := \big\{ x^* \in X^* \;\big|\; (x^*, -1) \in N\big((\bar{x}, \varphi(\bar{x})); \operatorname{epi} \varphi; (u, \mu)\big) \big\}.
\end{equation*}
Our defined subdifferential recovers this definition under $\mu=0$. Consequently, this minor divergence induces slight modifications in the expressions of Theorem \ref{thm: chain rule}.
\end{remark}

Similarly, the \emph{directional Mordukhovich singular subdifferential} 
$\partial^{\infty}\varphi(\bar{x}; \bar{u})$ of $\varphi$ at $\bar{x} \in \operatorname{dom}\varphi$ in direction $\bar{u}$ is defined by
\begin{equation*}
\begin{aligned}
\partial^{\infty}\varphi(\bar{x}; \bar{u})
&:= 
\big\{
x^* \in X^* ~\big|~
\exists\, \varepsilon_k \downarrow 0,~ t_k \downarrow 0,~ u_k \to \bar{u},~ \lambda_k \downarrow 0,~
x_k^* \overset{w^*}{\to} x^*,\\
&\quad\quad\quad\quad\quad\quad\text{with } \varphi(\bar{x} + t_k u_k) \to \varphi(\bar{x}) 
\text{ and } x_k^* \in \lambda_k \widehat{\partial}_{\varepsilon_k}\varphi(\bar{x} + t_k u_k)
\big\}. 
\end{aligned}
\end{equation*}

Furthermore, when $X$ is an Asplund space and $\varphi$ is lower semicontinuous, the $\varepsilon_k$-dependence in the definitions of $\partial \varphi$ and $\partial^\infty \varphi$ can be omitted \cite[Theorem 2.34]{Mordukhovich2006variational}. 

\subsection{Calculus of directional analysis}\label{subsec for preliminary directional results}

The primary objective of our analysis is to investigate the \emph{directional subdifferentials}.  
To achieve this, we establish a foundational framework centered on the properties of directional normal cones in infinite-dimensional settings. 
Specifically, in this subsection, we extend several fundamental directional calculus rules from finite-dimensional spaces to Asplund spaces. 

\begin{definition}
Let \( F : X \rightrightarrows Y \) be a set-valued mapping, and \( (\bar{x}, \bar{y}) \in \operatorname{gph} F \).
\begin{itemize}
    \item[(i)] \( F \) is said to be \emph{metrically subregular} at \( (\bar{x}, \bar{y}) \) if there exist a constant \( \kappa > 0 \) and a neighborhood \( U \subset X \) of \( \bar{x} \) such that: for any $x \in U$, we have 
    \begin{equation}\label{eq:metric_subregularity}
        \operatorname{dist}\bigl(x, F^{-1}(\bar{y})\bigr)
        \;\le\;
        \kappa\, \operatorname{dist}\bigl(\bar{y}, F(x)\bigr).
    \end{equation}
     Given a direction $u \in X$, $F$ is said to be \emph{metrically subregular in direction $u$} at $(\bar{x}, \bar{y})$ if there exists a directional neighborhood $\mathcal{V}(\bar x; u)$ such that \eqref{eq:metric_subregularity} holds for all $x \in \mathcal{V}(\bar{x};u)$.

\item[(ii)] 
\( F \) is said to be \emph{metrically regular} at \( (\bar{x}, \bar{y}) \) if there exist a constant \( \kappa > 0 \) and neighborhoods \( U \subset X \) of \( \bar{x} \), $V\subset Y$ of $\bar y$ such that: for any $x \in U,y\in V$, we have 
    \begin{equation}\label{eq:metric_regularity}
        \operatorname{dist}\bigl(x, F^{-1}(y)\bigr)
        \;\le\;
        \kappa\, \operatorname{dist}\bigl(y, F(x)\bigr).
    \end{equation}

\item[(iii)] 
\( F \) is said to be \emph{strongly metrically regular} at \( (\bar{x}, \bar{y}) \) if $F^{-1}$ has a Lipschitzian single-valued localization around $\bar y$ for $\bar x$ (this is equivalent to $F$ being metrically regular at $(\bar{x}, \bar{y})$ and $F^{-1}$ being locally single-valued around $(\bar{y}, \bar{x})$ \cite[Proposition~3G.1]{dontchev2009implicit}).   

\end{itemize}
\end{definition}

Classical sufficient conditions for the metric regularity of \eqref{eq: basic model of lower level} include Robinson's constraint qualification and the no nonzero abnormal multiplier constraint qualification. For general set-valued mappings, strong metric regularity can be established via the strict derivative criterion \cite[Theorem 4D.1]{dontchev2009implicit}. Notably, for the Karush--Kuhn--Tucker optimality system in nonlinear programming, metric regularity inherently implies strong metric regularity \cite[Section 4H]{dontchev2009implicit}. 

Recent developments have introduced refined sufficient conditions for the metric subregularity of \eqref{eq: basic model of lower level}, such as the first-order sufficient condition for metric subregularity and directional quasi-/pseudo-normality \cite{bai2019directional}. Furthermore, metric subregularity is automatically satisfied under specific geometric structures: according to Robinson's multifunction theory, it holds unconditionally when the system mapping is affine and the abstract constraint set is a finite union of polyhedral convex sets \cite{robinson1981some, bai2019directional}.

Utilizing \emph{metric subregularity}, we present the following extension of \cite[Theorem~3.1]{benko2019calculus}, which generalizes the calculus of directional normal cones to preimage sets from finite-dimensional settings to infinite-dimensional ones, under an assumption of Hadamard differentiability.

\begin{theorem}\label{thm: directional normal cones of preimage sets} 
Let \( X \) and \( Z \) be Asplund spaces, and \( Q \subset Z \) be closed.  
Consider a continuous mapping \( \phi : X \to Z \), and define
$C = \phi^{-1}(Q):= \{x\mid \phi(x)\in Q\}, ~F(x) := Q - \phi(x)$. 
Assume that \( F \) is metrically subregular at \( (\bar{x}, 0) \) in some direction \( h \in X \).  
Suppose further that \( \phi \) is Hadamard differentiable at \(\bar{x}\) in direction \( h \) with directional derivative $v := \phi'_H(\bar{x}; h)\in T(\phi(\bar x);Q)$. Then the following inclusion holds:
\begin{equation}\label{eq:dir_normal_preimage}
    N(\bar{x}; C; h) \;\subset\; D^*_N \phi(\bar{x}; h) \big( N(\phi(\bar{x}); Q; v) \big).
\end{equation}
\end{theorem}

\begin{proof}
Take any \( x^* \in N(\bar{x}; C; h) \).  
Then there exist sequences \( t_k \downarrow 0 \), \( h_k \to h \), and \( x_k^* \xrightarrow{w^*} x^* \) such that $x_k := \bar{x} + t_k h_k, ~x_k^* \in \hat{N}(x_k; C)$. 
By definition of $\hat N$, for each \( k \) and any \( \varepsilon > 0 \), there exists \( r_\varepsilon > 0 \) such that
\begin{equation}\label{eq:Frechet_ineq}
    \langle x_k^*, x - x_k \rangle 
    \;\leq\; \varepsilon \|x - x_k\|, 
    \quad \forall\, x \in C \cap B_{r_\varepsilon}(x_k).
\end{equation}

Since \( F \) is metrically subregular at \( (\bar{x}, 0) \) in direction \( h \), there exist a constant \( \kappa > 0 \) and a directional neighborhood \( \mathcal{V}(\bar x; h) \) of \( \bar x \) such that
\begin{equation}\label{eq:metric_subreg}
    \operatorname{dist}(x, C)
    = \operatorname{dist}(x, F^{-1}(0))
    \leq \kappa \operatorname{dist}(0, F(x))
    = \kappa \operatorname{dist}(\phi(x), Q),
    \quad \forall x \in \mathcal{V}(\bar x; h).
\end{equation}
For sufficiently large \( k \) and small \( \varepsilon > 0 \), shrink \( r_\varepsilon \) if necessary so that \( B_{r_\varepsilon}(x_k) \subset  \mathcal{V}(\bar{x}; h) \).  
We now claim that, for all \( x \in B_{r_\varepsilon/2}(x_k) \),
\begin{equation}\label{ineq claim}
    \varepsilon \|x - x_k\| - \langle x_k^*, x - x_k \rangle 
    + (\|x_k^*\| + \varepsilon)\, \kappa \operatorname{dist}(\phi(x), Q) \geq 0.
\end{equation}
Indeed, fix such an \( x \). If $x \in C$, \eqref{ineq claim} holds trivially by \eqref{eq:Frechet_ineq}. If $x \notin C$, for any arbitrarily small $\gamma > 0$, by the definition of distance, there exists an approximate projection \( \tilde{x} \in C \) satisfying
\[
    \|x - \tilde{x}\| < \operatorname{dist}(x, C) + \gamma \leq \kappa \operatorname{dist}(\phi(x), Q) + \gamma.
\]
By the continuity of $\phi$ and the fact that $x_k \in C$ (which yields $\operatorname{dist}(\phi(x_k), Q) = 0$), shrinking the radius to $r_\varepsilon/2$ guarantees that $\kappa \operatorname{dist}(\phi(x), Q)$ is sufficiently small. Hence, for a sufficiently small $\gamma$, we have $\|\tilde{x} - x_k\| \leq \|\tilde{x} - x\| + \|x - x_k\| < r_\varepsilon$, ensuring $\tilde{x} \in C \cap B_{r_\varepsilon}(x_k)$.
Applying \eqref{eq:Frechet_ineq} to $\tilde{x}$ and using the triangle inequality, we obtain:
\begin{align*}
    \langle x_k^*, x - x_k \rangle - \varepsilon \|x - x_k\|
    &= \langle x_k^*, x - \tilde{x} \rangle + \langle x_k^*, \tilde{x} - x_k \rangle - \varepsilon \|x - x_k\| \\
    &\leq \|x_k^*\| \|x - \tilde{x}\| + \varepsilon \|\tilde{x} - x_k\| - \varepsilon \|x - x_k\| \\
    &\leq \|x_k^*\| \|x - \tilde{x}\| + \varepsilon (\|\tilde{x} - x\| + \|x - x_k\|) - \varepsilon \|x - x_k\| \\
    &= (\|x_k^*\| + \varepsilon) \|x - \tilde{x}\| \\
    &< (\|x_k^*\| + \varepsilon) \big(\kappa \operatorname{dist}(\phi(x), Q) + \gamma\big).
\end{align*}
Since $\gamma > 0$ is arbitrary, sending $\gamma \downarrow 0$ yields \eqref{ineq claim}. 
Hence, for a sequence \( \varepsilon_k \downarrow 0 \), \( (x_k, \phi(x_k), \phi(x_k)) \) is a local minimizer of the problem: 
\begin{equation}\label{eq:aux_problem}
    \begin{aligned}
        \min_{x, y, q}\;\; & 
            \varepsilon_k \|x - x_k\| - \langle x_k^*, x - x_k \rangle 
            + (\|x_k^*\| + \varepsilon_k) \kappa \|y - q\| \\
        \text{s.t.}\;\; & (x, y, q) \in \operatorname{gph} \phi \times Q.
    \end{aligned}
\end{equation}

Applying the fuzzy optimality condition \cite[Lemma~2.32(ii)]{Mordukhovich2006variational}, we have 
\[
\|(x_{i,k}, y_{i,k}, q_{i,k}) - (x_k, \phi(x_k), \phi(x_k))\|\leq \eta_k \quad (i = 1, 2),
\]
and multipliers \( \xi_k^*\in B_{X^*}, v_k \in B_{Z^*} \), such that
\begin{equation}\label{eq:fuzzy_condition}
    \begin{aligned}
        &[x_k^* - \varepsilon_k \xi_k^*, -(\|x_k^*\| + \varepsilon_k)\kappa v_k] 
            \in \hat{N}\bigl((x_{2,k}, y_{2,k}); \operatorname{gph} \phi \bigr) + \eta_k (B_{X^*}\times B_{Z^*}), \\
        &(\|x_k^*\| + \varepsilon_k)\kappa v_k 
            \in \hat{N}(q_{2,k}; Q) + \eta_k B_{Z^*}, 
    \end{aligned}
\end{equation}
with \( \eta_k := t_k^2 \downarrow 0 \). 
The sequence \( \{(\|x_k^*\| + \varepsilon_k)\kappa v_k\} \) is bounded. Since \( Z \) is an Asplund space, passing to a subsequence if necessary, we may assume $(\|x_k^*\| + \varepsilon_k)\kappa v_k \xrightarrow{w^*} z^*$. 

From \eqref{eq:fuzzy_condition}, there exist \( z_k^* \in \hat{N}(q_{2,k}; Q) \) with $(\|x_k^*\| + \varepsilon_k)\kappa v_k \in z_k^* + \eta_k B_{Z^*}$. Since $\eta_k \downarrow 0$, this implies \( z_k^* \xrightarrow{w^*} z^* \). 
Moreover, since \( (x_{2,k}, y_{2,k}) \in \operatorname{gph}\phi \), we have $y_{2,k} = \phi(x_{2,k})$. The fuzzy condition gives $\|x_{2,k} - x_k\| \le \eta_k = t_k^2$, which implies $(x_{2,k} - \bar{x})/t_k = h_k + (x_{2,k} - x_k)/t_k \to h$.
By the Hadamard differentiability of \( \phi \) at \( \bar{x} \) in direction \( h \), we analyze the exact directional limits:
\begin{equation}\label{eq:estimate by hadamard diff_y}
    \frac{y_{2,k} - \phi(\bar{x})}{t_k} = \frac{\phi(x_{2,k}) - \phi(\bar{x})}{t_k} \to \phi'_H(\bar{x}; h) = v.
\end{equation}
Similarly, since $\|q_{2,k} - \phi(x_k)\| \le t_k^2$ and $\|y_{2,k} - \phi(x_k)\| \le t_k^2$, we obtain $\|q_{2,k} - y_{2,k}\| \le 2t_k^2$. This guarantees that the cross-term vanishes:
\begin{equation}\label{eq:estimate by hadamard diff_q}
    \frac{q_{2,k} - \phi(\bar{x})}{t_k} = \frac{q_{2,k} - y_{2,k}}{t_k} + \frac{y_{2,k} - \phi(\bar{x})}{t_k} \to 0 + v = v.
\end{equation}
Since \( q_{2,k} \in Q \) and $z_k^* \xrightarrow{w^*} z^*$, the limit in \eqref{eq:estimate by hadamard diff_q} verifies that \( z^* \in N(\phi(\bar{x}); Q; v) \). 
Finally, combining \eqref{eq:fuzzy_condition} with \( x_k^* - \varepsilon_k \xi_k^* \xrightarrow{w^*} x^* \) and \eqref{eq:estimate by hadamard diff_y}, we obtain
\[
(x^*, -z^*) \in N\bigl((\bar{x}, \phi(\bar{x})); \operatorname{gph} \phi; (h, v)\bigr),
\]
that is, $x^* \in D^*_N \phi(\bar{x}; h)(z^*)$. 
Therefore $N(\bar{x}; C; h) \subset D^*_N \phi(\bar{x}; h) \big( N(\phi(\bar{x}); Q; v) \big)$, which completes the proof. 
\end{proof}

In Theorem~\ref{thm: directional normal cones of preimage sets}, we require that \( v = \phi'_H(\bar{x}; h) \), whereas in the finite-dimensional case one typically assumes \( v \in D\phi(\bar{x})(h) \) and $D$ denotes the graphical derivative operator; see \cite[Theorem~3.1]{benko2019calculus}. 
This distinction arises because, in infinite-dimensional settings, one cannot guarantee the existence of a norm convergent subsequence of $\left\{ \frac{\phi(\bar{x} + t_k u_k) - \phi(\bar{x})}{t_k} \right\}$, 
which is necessary to ensure \eqref{eq:estimate by hadamard diff_q} holds. Therefore, we need Hadamard differentiability of \( \phi \) to ensure the existence of the limit.  

Next, we present a directional version of the scalarization formula for $D_N^*$, which generalizes the classical non-directional result \cite[Theorem~3.28]{Mordukhovich2006variational}. 
\begin{definition}[Weak* Strict Lipschitzian {\cite[Definition~3.25]{Mordukhovich2006variational}}] \label{defn:weak_strict_lip}
Let $\phi: X \to Y$ be a single-valued mapping between Banach spaces that is Lipschitz continuous around $\bar{x}$. We say that $\phi$ is \emph{weak* strictly Lipschitzian} at $\bar{x}$ if for any direction $u \in X$ and any sequences $x_k \to \bar{x}$, $t_k \downarrow 0$, and any bounded sequence $y_k^* \xrightarrow{w^*} 0$ in $Y^*$, it holds that $\langle y_k^*, y_k \rangle \to 0$ where
    \begin{equation*}
        y_k := \frac{\phi(x_k + t_k u) - \phi(x_k)}{t_k}, \quad k \in \mathbb{N}. 
    \end{equation*}
\end{definition}

Notably, this class includes Fredholm integral operators with Lipschitzian kernels \cite[Chap.~5]{zeidler2012applied}, which are important in applications to optimal control 
\cite[p.~288]{Mordukhovich2006variational}. When $\dim Y < \infty$, it is equivalent to local Lipschitz continuity. The following scalarization theorem extends the classical non-directional framework established in \cite[Theorem~3.28]{Mordukhovich2006variational}.  

\begin{theorem} \label{thm:scalarization}
Let $X,Y$ be Asplund spaces. Suppose that the mapping $\phi: X \to Y$ is locally Lipschitzian and weak* strictly Lipschitzian at $\bar{x} \in X$. If $\phi$ is Hadamard differentiable at $\bar{x}$ in a given direction $\bar{u} \in X$, then for any $y^* \in Y^*$, we have: 
\begin{equation} \label{eq:scalarization_formula}
    D_N^* \phi(\bar{x}; \bar{u})(y^*) = \partial \langle y^*, \phi \rangle (\bar{x}; \bar{u}).
\end{equation}
\end{theorem}

\begin{proof}
We only need to establish the inclusion $D_N^* \phi(\bar{x}; \bar{u})(y^*) \subset \partial \langle y^*, \phi \rangle (\bar{x}; \bar{u})$, as the converse inclusion follows immediately from \cite[Theorem~1.90]{Mordukhovich2006variational}.  
Let $x^* \in D_N^* \phi(\bar{x}; \bar{u})(y^*)$. By the definition, there exist sequences $\varepsilon_k \downarrow 0$, $t_k \downarrow 0$, $u_k \to \bar{u}$, and $x_k := \bar{x} + t_k u_k$ such that
\begin{equation*}
    (x_k^*, y_k^*) \xrightarrow{w^*} (x^*, y^*) \quad \text{and} \quad (x_k^*, -y_k^*) \in \widehat{N}((x_k, \phi(x_k)); \mathrm{gph}\, \phi).
\end{equation*}

By the definition of $\widehat{N}$, for $\gamma_k := t_k^2 \downarrow 0$, there exists a neighborhood $U_k$ of $x_k$ such that
\begin{equation*}
    \langle x_k^*, x - x_k \rangle - \langle y_k^*, \phi(x) - \phi(x_k) \rangle \le (\gamma_k + \varepsilon_k) (\|x - x_k\| + \|\phi(x) - \phi(x_k)\|), \quad \forall x \in U_k.
\end{equation*}
Using the Lipschitzian property of $\phi$ with modulus $\ell$, we obtain
\begin{equation} \label{eq:w_lip_estimate}
    \langle x_k^*, x - x_k \rangle - \langle y_k^*, \phi(x) - \phi(x_k) \rangle \le (\gamma_k + \varepsilon_k)(1 + \ell) \|x - x_k\|.
\end{equation}

Let $\tilde{\varepsilon}_k := (\gamma_k + \varepsilon_k)(1 + \ell) \downarrow 0$, we have $x_k^* \in \widehat{\partial}_{\tilde{\varepsilon}_k} \langle y_k^*, \phi \rangle (x_k) = \widehat{\partial}_{\tilde{\varepsilon}_k} [\langle y^*, \phi \rangle + \langle y_k^* - y^*, \phi \rangle] (x_k)$. According to the fuzzy sum rule in Asplund spaces \cite[Theorem~2.33(b)]{Mordukhovich2006variational}, there exist $x_{i,k} \in x_k + \gamma_k B_X$ for $i=1, 2$, such that
\begin{equation*}
    x_{1,k}^* \in \widehat{\partial} \langle y^*, \phi \rangle (x_{1,k}), \quad x_{2,k}^* \in \widehat{\partial} \langle y_k^* - y^*, \phi \rangle (x_{2,k}), \quad \|x_{1,k}^* + x_{2,k}^* - x_k^*\| \le \gamma_k+\tilde{\varepsilon}_k.
\end{equation*}
Since $\gamma_k = t_k^2$, the points $x_{1,k}$ and $x_{2,k}$ converge to $\bar{x}$ in the direction $\bar{u}$. For $\phi$, we have $\widehat{\partial} \langle z^*, \phi \rangle (x) = \widehat{D}^* \phi(x)(z^*)$, for any $z^*\in Y^*$ \cite[(3.37)]{Mordukhovich2006variational}. 
Then because $y_k^* - y^* \xrightarrow{w^*} 0$, it follows from \cite[Lemma~3.27(i)]{Mordukhovich2006variational} that $x_{2,k}^* \xrightarrow{w^*} 0$. Consequently, $x_{1,k}^* \xrightarrow{w^*} x^*$, which implies $x^* \in \partial \langle y^*, \phi \rangle (\bar{x}; \bar{u})$. 
\end{proof}

Utilizing the preimage theorem, we establish a directional chain rule for subdifferentials in infinite dimensions. 
While a finite-dimensional version is available in \cite[Theorem~4.1]{benko2019calculus}, our result specializes this case by restricting the direction $0$ of the function value (see Remark~\ref{rem:defn_of_other_directional_subdifferential}).   

\begin{theorem}[Chain Rule]\label{thm: chain rule}
Let \( X, Z \) be Asplund spaces, and let \( \phi : X \to Z \) be continuous.  
Let \( g : Z \to \overline{\mathbb{R}} \) be finite at \( \phi(\bar{x}) \) and lower semicontinuous, and define \( \psi := g \circ \phi \).  
Given a direction \( (h, 0) \in X \times \mathbb{R} \), suppose the mapping
\[
    F : X \times \mathbb{R} \rightrightarrows Z \times \mathbb{R}, 
    \quad F(x, \alpha) := \operatorname{epi} g - (\phi(x), \alpha),
\]
is metrically subregular at \( ((\bar{x}, \psi(\bar{x})), (0,0)) \) in direction \( (h, 0) \), and $\phi$ is Hadamard differentiable in $h$ at $\bar x$ with \( v = \phi'_H(\bar{x}; h) \).  
Then
\begin{equation}\label{eq:chain_rule_unscalarized}
    \partial \psi(\bar{x}; h)
    \;\subset\;
    D^*_N \phi(\bar{x}; h) \big( \partial g(\phi(\bar{x}); v)\big).
\end{equation}
Furthermore, if $\phi$ is weak* strictly Lipschitzian at $\bar{x}$, then 
\begin{equation}\label{eq:chain_rule}
    \partial \psi(\bar{x}; h)
    \;\subset\;\bigcup\limits_{z^*\in \partial g(\phi(\bar{x}); v)} \partial \langle z^*, \phi\rangle(\bar{x}; h).
\end{equation}
\end{theorem}

\begin{proof}
Take any $x^*\in \partial \psi(\bar x; h)$. By the definition of the directional limiting subdifferential, and for the epi-direction $(h, 0)$, we have  
\( (x^*, -1) \in N((\bar{x}, \psi(\bar{x})); \operatorname{epi} \psi; (h, 0)) \).  
For any \( (x, \alpha) \in \operatorname{epi} \psi \), we have a bijective correspondence with \( (\phi(x), \alpha) \), so that
\[
    \operatorname{epi} \psi = \Psi^{-1}(\operatorname{epi} g),
    \quad \text{where } \Psi(x, \alpha) := (\phi(x), \alpha).
\]
Note that the mapping $\Psi$ is Hadamard differentiable at $(\bar{x}, \psi(\bar{x}))$ in the direction $(h, 0)$, yielding $\Psi'_H(\bar{x}, \psi(\bar{x}); (h, 0)) = (\phi'_H(\bar{x}; h), 0) = (v, 0)$.
Applying Theorem~\ref{thm: directional normal cones of preimage sets} with \( \Psi \) in place of \( \phi \), \( \operatorname{epi} g \) in place of \( Q \) and $C:=\{(x,\alpha) \mid (\phi(x),\alpha)\in \text{epi}g)\}$ therein,   
we deduce that
\begin{equation}\label{eq:preimage_chain}
    (x^*, -1) 
    \in D^*_N \Psi((\bar{x},\psi(\bar x)); (h, 0)) \,
    \big( N((\phi(\bar{x}), \psi(\bar{x})); \operatorname{epi} g; (v, 0)) \big).
\end{equation}

Hence, there exist multipliers \( (y^*, \beta) \) such that
\begin{equation}\label{eq:normals_gphPsi}
    (x^*, -1, -y^*, -\beta)
    \in N\bigl((\bar{x}, \psi(\bar{x}), \phi(\bar{x}), \psi(\bar{x}));
               \operatorname{gph} \Psi; (h, 0, v, 0)\bigr).
\end{equation}
Define $\Omega_\Psi
    := \{ (a, b, c, d) \in X \times Z \times \mathbb{R}^2 
        \mid (a, c, b, d) \in \operatorname{gph} \Psi \}$. 
Then, by the definition of the directional normal cone, condition~\eqref{eq:normals_gphPsi} is equivalent to
\begin{equation}\label{eq:normals_omega}
    (x^*, -y^*, -1, -\beta)
    \in N\bigl((\bar{x}, \phi(\bar{x}), \psi(\bar{x}), \psi(\bar{x}));
               \Omega_\Psi; (h, v, 0, 0)\bigr).
\end{equation}
By Proposition~\ref{prop of directional normal cones of product of separated sets},  
and noting that $\Omega_\Psi = \operatorname{gph} \phi \times \operatorname{gph} \operatorname{Id}|_{\psi}$, 
where \( \operatorname{Id}|_{\psi} \) denotes the identity mapping restricted to the range of \( \psi \),  
we can decompose the directional normal cone:
\begin{equation}\label{eq:decomposition of normal cone in chain rule}
    \begin{aligned}
    &N\bigl((\bar{x}, \phi(\bar{x}), \psi(\bar{x}), \psi(\bar{x}));
        \operatorname{gph} \phi \times 
        \operatorname{gph} \operatorname{Id}|_{\psi}; 
        (h, v, 0, 0)\bigr) \\
    &\subset
    N((\bar{x}, \phi(\bar{x})); \operatorname{gph} \phi; (h, v))
    \times
    N((\psi(\bar{x}), \psi(\bar{x})); 
      \operatorname{gph} \operatorname{Id}|_{\psi}; (0, 0)).
    \end{aligned}
\end{equation}
Thus we have
\begin{equation}
    \begin{split}
        &(x^*, -y^*) \in N((\bar{x}, \phi(\bar{x})); \operatorname{gph} \phi; (h, v)),\\
&(-1, -\beta) \in 
N((\psi(\bar{x}), \psi(\bar{x})); 
  \operatorname{gph} \operatorname{Id}|_{\psi}; (0, 0)). 
    \end{split}
\end{equation}
Because the direction is identically zero, the latter condition is evaluated on the standard normal cone to the identity graph, which strictly enforces $-1 - \beta = 0$, leading to \( \beta = -1 \).  
Consequently, $(y^*, -1) \in N((\phi(\bar{x}), \psi(\bar{x})); \operatorname{epi} g; (v, 0))$, meaning \( y^* \in \partial g(\phi(\bar{x}); v) \). It follows that $x^* \in D^*_N \phi(\bar{x}; h)(y^*)$. 
This proves the unscalarized inclusion \eqref{eq:chain_rule_unscalarized}. 
The scalarized inclusion \eqref{eq:chain_rule} then follows immediately by Theorem \ref{thm:scalarization}. 
\end{proof}

In addition, we employ estimates of the \emph{singular subdifferential} of value functions to investigate their directional Lipschitz properties. 
To this end, we establish a result, which generalizes the finite-dimensional version presented in \cite[Proposition~2.3]{Bai2023Directional}. 
Before proceeding, we recall the notions of \emph{directional local lower semicontinuity} and \emph{directional Lipschitz continuity}.  

\begin{definition}
Let $\phi:X\to \mathbb{R}\cup\{\infty\}$ with $\bar x\in \operatorname{dom}\phi$. 
We say that $\phi$ is \emph{locally lower semicontinuous} at $\bar x$ in direction $u$ if there exist $\varepsilon>0$ and $\delta>0$ such that the following set is closed: 
\[
\operatorname{epi}\phi \cap [~ \overline{\mathcal{V}_{\varepsilon,\delta}(\bar x; u)} \times [\phi(\bar x)-\varepsilon,\phi(\bar x)+\varepsilon]~]. 
\]
\end{definition}

The notion of directional lower semicontinuity is stronger than that of lower semicontinuity, and indeed, directional continuity does not imply local lower semicontinuity; see \cite[Example~2.1]{Bai2023Directional}. 
We next recall the concept of directional Lipschitz continuity introduced in \cite{benko2019calculus}.

\begin{definition}[Directional Lipschitz continuity]
Let $\phi:X\rightarrow Y$. 
We say that $\phi$ is \emph{Lipschitz continuous} around $\bar x$ in direction $u$ if there exist a constant $L\geq 0$ and a directional neighborhood $\mathcal{V}(\bar x; u)$ such that 
\[
\|\phi(x)-\phi(z)\|\leq L\|x-z\|, \quad \forall\, x,z\in \mathcal{V}(\bar x;u).
\]
\end{definition}

The following proposition extends the classical relationship between the singular subdifferential and local Lipschitz continuity (cf. \cite[Theorem 3.52]{Mordukhovich2006variational}) to the directional framework. 

\begin{proposition}\label{prop:singular_subdiff_Lipschitz}
Let $X$ be an Asplund space and $\phi: X \to \mathbb{R} \cup \{\infty\}$ be locally lower semicontinuous and continuous at $\bar{x}$ in some direction $u \neq 0$, with $\phi(\bar{x})$ finite. Then the following statements hold:
\begin{itemize}
    \item[\rm(i)] If $\phi$ is Lipschitz continuous at $\bar{x}$ in the direction $u$, then $\partial^{\infty} \phi(\bar{x}; u) = \{0\}$.
    \item[\rm(ii)] If $\partial^{\infty} \phi(\bar{x}; u) = \{0\}$ and $\phi$ satisfies the following condition: for any $x_k\xrightarrow{u} \bar x$, with $\phi(x_k)\to \phi(\bar x)$ and $\lambda_k \downarrow 0$ with $x_k^*\in \lambda_k \hat{\partial} \phi(x_k)$, one has 
    \begin{equation*}
        \begin{split}
            x_k^*\xrightarrow{w^*} 0 \implies x_k^* \to 0,  
        \end{split}
    \end{equation*}
    then $\phi$ is Lipschitz continuous at $\bar{x}$ in the direction $u$. 
\end{itemize}
\end{proposition}

\begin{proof}
If $\phi$ is Lipschitz continuous around $\bar{x}$ in the direction $u$, the forward implication $\partial^{\infty} \phi(\bar{x}; u) = \{0\}$ follows from \cite[Corollary~5.9]{long2017calculus}.

Conversely, assume that $\partial^{\infty} \phi(\bar{x}; u) = \{0\}$. We clarify that the Fr\'echet subdifferential $\widehat{\partial} \phi$ is uniformly bounded on a localized directional neighborhood of $\bar{x}$. Suppose by contradiction that it fails. Then there exist sequences $x_k \in \mathcal{V}_{1/k, 1/k}(\bar{x};u)$ and $x_k^* \in \widehat{\partial} \phi(x_k)$ such that $M_k := \|x_k^*\| \to \infty$ as $k \to \infty$. 
By the directional continuity of $\phi$ at $\bar x$, we also have $\phi(x_k) \to \phi(\bar x)$. 

Define $\lambda_k := M_k^{-1} \downarrow 0$, which yields $\|\lambda_k x_k^*\| = 1$ for all $k \in \mathbb{N}$. Since $X$ is an Asplund space, extracting a subsequence if necessary, we can find some $x^* \in X^*$ such that $\lambda_k x_k^* \xrightarrow{w^*} x^*$. Then it holds that $x^* \in \partial^{\infty} \phi(\bar{x}; u)$. Under the hypothesis that $\partial^{\infty} \phi(\bar{x}; u) = \{0\}$, this forces the weak* limit to be the origin, i.e., $x^* = 0$, which implies $\lambda_k x_k^* \xrightarrow{w^*} 0$. It then follows that $\lambda_k x_k^* \to 0$ in norm, which leads to a contradiction. 

This contradiction establishes that the Fr\'echet subgradients must be uniformly bounded by some constant $M > 0$ on a directional neighborhood $\mathcal{V}(\bar{x}; u)$. By employing the variational error bound arguments restricted to the directional wedge (e.g. \cite[Theorem 3.52]{Mordukhovich2006variational}), this uniform boundedness ensures that $\phi$ is directionally Lipschitz continuous at $\bar{x}$ in the direction $u$. 
\end{proof}

In Asplund spaces, the condition in (ii) of Proposition \ref{prop:singular_subdiff_Lipschitz} can be viewed as a directional counterpart of sequentially normally epi-compact (SNEC) functions \cite[Definition 1.116, Corollary 2.39]{Mordukhovich2006variational}. This property is naturally satisfied if the mapping possesses local Lipschitzian behavior, as established in \cite[Theorem 1.26, Corollary 1.81]{Mordukhovich2006variational}. 

We now present a directional \emph{sum rule} for subdifferentials of functions expressible as the sum of a locally Lipschitz function and a lower semicontinuous function.  
We extend the classical nonsmooth sum rule of \cite[Theorem~2.33(c)]{Mordukhovich2006variational} to the \emph{directional} setting, and generalize the finite-dimensional result of \cite[Corollary~4.5]{benko2019calculus} to Asplund spaces.  
The result plays a central role in subsequent developments involving directional subdifferentials.

\begin{theorem}[Sum Rule]\label{thm: sum rule}
Let \( X \) be an Asplund space.  
Suppose \( \phi_1 : X \to \mathbb{R} \) is locally Lipschitz around \( \bar{x} \), and \( \phi_2 : X \to \mathbb{R} \cup \{\infty\} \) is lower semicontinuous at \( \bar{x} \), with \( \phi_i(\bar{x}) < \infty \) for \( i = 1,2 \).  
Then, for any direction \( \bar{u} \in X \),
\[
    \partial(\phi_1 + \phi_2)(\bar{x}; \bar{u})
    \subset \partial \phi_1(\bar{x}; \bar{u}) + \partial \phi_2(\bar{x}; \bar{u}).
\]
\end{theorem}

\begin{proof}
Let \( x^* \in \partial(\phi_1 + \phi_2)(\bar{x}; \bar{u}) \).  
By definition, there exist sequences \( t_k \downarrow 0 \), \( \varepsilon_k \downarrow 0 \), \( u_k \to \bar{u} \), and \( x_k = \bar{x} + t_k u_k \) such that
\[
    \phi_1(x_k) + \phi_2(x_k) \to \phi_1(\bar{x}) + \phi_2(\bar{x}), 
    \quad 
    x_k^* \in \hat{\partial}_{\varepsilon_k}(\phi_1 + \phi_2)(x_k),
    \quad 
    x_k^* \xrightarrow{w^*} x^*.
\]

Since \( X \) is Asplund, the fuzzy sum rule 
\cite[Theorem~2.33(b)]{Mordukhovich2006variational}    
with \( \gamma_k = t_k^2 \) provides sequences 
\( x_{i,k} \to \bar{x} \), \( \phi_i(x_{i,k}) \to \phi_i(\bar{x}) \), and 
\( x_{i,k}^* \in \hat{\partial} \phi_i(x_{i,k}) \) for \( i = 1,2 \),  
such that \( x_{i,k} \in x_k + \gamma_k B_X \) and
\begin{equation}\label{eq:fuzzy-sum}
    \|x_k^* - x_{1,k}^* - x_{2,k}^*\| \le \varepsilon_k + \gamma_k.
\end{equation}

Because \( x_k^* \xrightarrow{w^*} x^* \), the sequence \( \{x_k^*\} \) is bounded.  
By the local Lipschitz continuity of \( \phi_1 \),  
\cite[Proposition~1.85(ii)]{Mordukhovich2006variational} implies that \( \{x_{1,k}^*\} \) is bounded; hence, by \eqref{eq:fuzzy-sum}, so is \( \{x_{2,k}^*\} \).  
Since $X$ is an Asplund space, it ensures the existence of \( x_i^* \in X^* \) such that, possibly along a subsequence,  
\( x_{i,k}^* \xrightarrow{w^*} x_i^* \) for \( i = 1,2 \).  
Then, by \cite[Theorem~1.89]{Mordukhovich2006variational}, we have \( x_i^* \in \partial \phi_i(\bar{x}) \), and \eqref{eq:fuzzy-sum} gives \( x^* = x_1^* + x_2^* \).

Furthermore, since \( x_{i,k} = \bar{x} + t_k u_k + \gamma_k e_{i,k} \) with \( e_{i,k} \in B_X \), we obtain
\[
    \frac{x_{i,k} - \bar{x}}{t_k} = u_k + t_k e_{i,k} \to \bar{u}, \quad i = 1,2.
\]
So \( x_i^* \in \partial \phi_i(\bar{x}; \bar{u}) \) for \( i = 1,2 \), which yields $x^* \in \partial \phi_1(\bar{x}; \bar{u}) + \partial \phi_2(\bar{x}; \bar{u})$, 
completing the proof. 
\end{proof}

\section{Subdifferentials of the Value Function}

In this section, we study the \emph{directional subdifferentials} of the value function \( V \). 
By employing the calculus developed in the preceding section, we derive estimates for both the directional and singular subdifferentials of \( V \), thereby extending the finite-dimensional results to Asplund spaces. 
The main challenge arises from the lack of compactness in infinite-dimensional settings, which we overcome through the use of suitable variational conditions in different cases. 

\subsection{Continuity of the Value Function}
In subdifferential analysis, it is common to consider lower semicontinuous functions. 
Therefore, before analyzing subdifferentials, we first examine the continuity of \( V \). 
To this end, we introduce the concept of \(\tau\)-restricted inf-compactness for the value function \( V \), which extends the notion of inf compactness to settings involving weaker topologies or directional structures.

\begin{definition}[\(\tau\)-Restricted Inf-Compactness]\label{defn: restricted inf compactness}
Let \( V : X \to \mathbb{R} \cup \{\infty\} \) be the value function in \eqref{eq: basic model of lower level}. 
Let \( \tau \) be a locally convex Hausdorff linear topology on the decision space \( Y \). 
We say that \( V \) is \emph{\(\tau\)-restricted inf-compact} at \( \bar{x} \in X \) if \( V(\bar{x}) < \infty \) and there exist \( \varepsilon > 0 \), a \( \tau \)-sequentially compact set \( \Omega \subset Y \), and a neighborhood \( U \subset X \) of \( 0 \) such that for all \( x \in \bar{x} + U \) satisfying \( V(x) < V(\bar{x}) + \varepsilon \), we have $S(x) \cap \Omega \neq \emptyset$. 

If \( \tau \) is taken to be the norm topology on \( Y \), we omit the reference to \( \tau \) and simply say that \( V \) is \emph{restricted inf-compact} at \( \bar{x} \). 
Moreover, when \( \tau \) is the norm topology on $Y$, and the neighborhood $\bar{x} + U$ is replaced by a directional neighborhood \( \mathcal{V}(\bar{x};\bar{u}) \) in direction $\bar{u} \in X$ (along with a norm-compact set denoted by \( \Omega_{\bar{u}} \subset Y \)), we say that \( V \) is \emph{directionally restricted inf-compact} at $\bar{x}$ in direction \( \bar{u} \). 
When \( \bar{u} = 0 \), this reduces to the standard restricted inf-compactness defined above.
\end{definition}

To illustrate the concept of \(\tau\)-restricted inf-compactness in the weak topology, we provide the following example. 

\begin{example}\label{example weak infcompact}
Let \( X,Y \) be reflexive Banach spaces.   
Recall that any continuous convex function \( h : X \to \mathbb{R} \) is weakly sequentially lower semicontinuous; see \cite[Theorem~1.18]{hinze2008optimization}.  
A set \( A \subset X \) is called \emph{weakly sequentially compact} (w.s.c.) if every sequence in \( A \) admits a subsequence converging weakly to a point in \( A \).  
It is well known that closed convex sets are weakly sequentially closed \cite[Theorem~1.16]{hinze2008optimization}; any bounded, closed, and convex subset of a reflexive Banach space is weakly sequentially compact \cite[Theorem~1.17]{hinze2008optimization}.

A set-valued mapping \( \Gamma : X \rightrightarrows Y \) is said to be \emph{locally uniformly weakly sequentially compact} at \( \bar{x} \in X \) if there exists a neighborhood \( U \) of \( \bar{x} \) such that $\bigcup_{x \in U} \Gamma(x)$ 
is weakly sequentially compact.  
If the neighborhood \( U \) is replaced by a directional neighborhood \( \mathcal{V}(\bar{x};u) \), we say that \( \Gamma \) is \emph{locally uniformly weakly sequentially compact in direction \( u \)}.  Then 
under these conditions, \( V \) satisfies the (directional) restricted inf-compactness property.
\end{example}

Based on the restricted inf-compactness of \( V \), we establish its lower semicontinuity, which extends \cite[Proposition~3.1]{Bai2023Directional} to infinite-dimensional settings.  

\begin{theorem}[Lower Semicontinuity of the Value Function]\label{theorem of lsc for V}
Suppose \( V \) is \(\tau\)-restricted inf-compact at \( \bar{x} \), and the graph of the feasible mapping \( \Gamma \) is norm-\(\tau\)-closed and $f$ is norm-$\tau$ lower semicontinuous.   
Then \( V \) is lower semicontinuous at \( \bar{x} \).  
Moreover, if \( V \) is directionally restricted inf-compact at \( \bar{x} \) in direction \( \bar{u} \), then \( V \) is lower semicontinuous at $\bar x$ in direction \( \bar{u} \).
\end{theorem}

\begin{proof}
Let \( \{x_k\} \subset X \) be any sequence such that \( x_k \to \bar{x} \) and $\liminf_{x \to \bar{x}} V(x) = \lim_{k \to \infty} V(x_k)$. 
Assume, to the contrary, that $\lim_{k \to \infty} V(x_k) < V(\bar{x})$. 

Then, by the \(\tau\)-restricted inf-compactness of \( V \), for all sufficiently large \( k \) we have \( S(x_k) \neq \emptyset \); that is, there exist \( y_k \in S(x_k) \) such that \( V(x_k) = f(x_k, y_k) \).  

Because \( \Omega \) in Definition~\ref{defn: restricted inf compactness} is \(\tau\)-compact, and \( \{y_k\} \subset \Omega \), we may, without loss of generality, assume that \( y_k \xrightarrow{\tau} \bar{y} \in \Omega \).  
By the norm-\(\tau\)-closedness of \( \operatorname{gph}\Gamma \), it follows that \( \bar{y} \in \Gamma(\bar{x}) \).  
Hence, by the lower semicontinuity of $f$, $\lim_{k \to \infty} f(x_k, y_k) \geq f(\bar{x}, \bar{y})$, and consequently,                                
\[
V(\bar{x}) > \lim_{k \to \infty} V(x_k) = \lim_{k \to \infty} f(x_k, y_k) 
    \geq f(\bar{x}, \bar{y}) \geq V(\bar{x}),\]
which yields a contradiction. 
Therefore, \( V \) must be lower semicontinuous at \( \bar{x} \).

For the directional case, consider a sequence \( x_k = \bar{x} + t_k u_k \) with \( t_k \downarrow 0 \) and \( u_k \to \bar{u} \).  
Repeating the same argument under the assumption of directional restricted inf-compactness (i.e., replacing the neighborhood by some \( \mathcal{V}(\bar{x}; \bar{u}) \)), we obtain the directional lower semicontinuity of \( V \) in direction \( \bar{u} \). 
\end{proof}

We will need the following concept of directional solution.
\begin{definition}[Directional Solution {\cite[Definition 4.5]{bai2022directional}}]
    The optimal solution in direction $u$ is defined by
    \begin{equation*}
        S(\bar{x}; u) = \left\{ y \in S(\bar{x}) \;\middle|\; \exists t_k \downarrow 0, \, u_k \to u, \, y_k \to y, \, y_k \in S(\bar{x} + t_k u_k) \right\}.
    \end{equation*}
\end{definition}

\subsection{Directional Subdifferentials of the Value Function}

In this subsection, we employ variational analysis to derive directional subdifferentials of value functions in Asplund spaces. 
These results extend analogous findings in finite-dimensional cases \cite{benko2019calculus,Bai2023Directional} and provide stronger conclusions in Asplund spaces compared to those established for general Banach spaces \cite{long2017calculus}. 

We introduce the notion of \emph{strongly directional inner semicompactness}. This concept serves as a key tool for extending subdifferential results from finite-dimensional settings to infinite-dimensional ones. 

\begin{definition}[\cite{long2017calculusb}, Definition~1.1]\label{defn: directional strong inner semicompactness}
Let \( F: X \rightrightarrows Y \) be a set-valued mapping. 
\( F \) is said to be strongly directional inner semicompact at \( (\bar{x}, \bar{y}) \in \operatorname{gph} F \) in the direction \( (\bar{u},\bar v) \in X\times Y \) if, for any sequences \( t_k \downarrow 0 \) and \( u_k \to \bar{u} \) with \( F(\bar{x} + t_k u_k) \neq \emptyset \), there exist \( y_k \in F(\bar{x} + t_k u_k) \) such that the sequence \( \left\{ \frac{y_k - \bar{y}}{t_k} \right\} \) admits a convergent subsequence to $\bar v$.
\end{definition}

\begin{remark}\label{rem:inner_calm}
    A set-valued mapping $F: X \rightrightarrows Y$ is called directionally inner calm at $(\bar{x}, \bar{y}) \in \operatorname{gph} F$ in a given direction $u \in X$ if there exist a constant $\kappa > 0$ and a directional neighborhood $\mathcal{V}(\bar{x}; u)$ such that \cite[Definition 2.2]{benko2019calculus}
    \begin{equation}\label{eq:inner_calm}
        \bar{y} \in F(x) + \kappa \|x - \bar{x}\| B_Y, \quad \forall x \in \mathcal{V}(\bar{x}; u).
    \end{equation}

    Consequently, for any sequences $t_k \downarrow 0$ and $u_k \to u$, by setting $x_k := \bar{x} + t_k u_k$, \eqref{eq:inner_calm} ensures the existence of $y_k \in F(x_k)$ satisfying $\|y_k - \bar{y}\| \leq \kappa t_k \|u_k\|$. Dividing both sides by $t_k$, we deduce that the difference quotient sequence $\left\{ \frac{y_k - \bar{y}}{t_k} \right\}$ is bounded by $\kappa \|u\|+1$. When $Y$ is finite-dimensional, this boundedness guarantees the existence of a convergent subsequence and a direction $v \in Y$ such that $\frac{y_{k_j} - \bar{y}}{t_{k_j}} \to v$. Thus, when $Y$ is a finite-dimensional space, directional inner calmness naturally implies strongly directional inner semicompactness. 
\end{remark}

In what follows, we introduce a concept and establish a proposition to verify the strongly directional inner semicompactness of a set-valued mapping. 

\begin{definition}\label{defn:directionally_single_valued}
    We say that $F: X \rightrightarrows Y$ has a Hadamard single-valued localization at $(\bar{x}, \bar{y})$ in the direction $\bar{u}$ if there exist a directional neighborhood $\mathcal{V}(\bar{x}; \bar{u})$, a neighborhood $U_{\bar{y}}$ of $\bar{y}$, and a mapping $\varphi: X \to Y$ with $\varphi(\bar{x}) = \bar{y}$ such that $\{\varphi(x)\} = F(x) \cap U_{\bar{y}}$ for any $x \in \mathcal{V}(\bar{x}; \bar{u})$, and $\varphi$ is Hadamard differentiable at $\bar{x}$ in the direction $\bar{u}$ and locally Lipschitzian in $\mathcal{V}(\bar{x}; \bar{u})$. 
\end{definition}

\begin{proposition}\label{prop: single_valued_implies_semicompact}
    If $F: X \rightrightarrows Y$ has a Hadamard single-valued localization at $(\bar{x}, \bar{y})$ in the direction $\bar{u}$, then $F$ is strongly directionally inner semicompact at $(\bar x,\bar y)$ in the direction $(\bar u,\bar v)$ for some $\bar v\in Y$. 
\end{proposition}

\begin{proof}
    By Definition \ref{defn:directionally_single_valued}, there exist a directional neighborhood $\mathcal{V}(\bar{x}; \bar{u})$ and a mapping $\varphi: X \to Y$ with $\varphi(\bar{x}) = \bar{y}$ such that $\{\varphi(x)\} = F(x) \cap U_{\bar{y}}$ for any $x \in \mathcal{V}(\bar{x}; \bar{u})$. Moreover, $\varphi$ is Hadamard directionally differentiable at $\bar{x}$ in the direction $\bar{u}$. 
    
    Consequently, for any sequences $t_k \downarrow 0$ and $u_k \to \bar{u}$, setting $x_k := \bar{x} + t_k u_k$, we have $x_k \in \mathcal{V}(\bar{x}; \bar{u})$ for all sufficiently large $k$. This allows us to uniquely select $y_k := \varphi(x_k) \in F(x_k)$. By the Hadamard directional differentiability of $\varphi$, we directly obtain:
    \begin{equation*}
        \lim_{k \to \infty} \frac{y_k - \bar{y}}{t_k} = \lim_{k \to \infty} \frac{\varphi(\bar{x} + t_k u_k) - \varphi(\bar{x})}{t_k} = \varphi'_H(\bar{x}; \bar{u}).
    \end{equation*}
    Setting $\bar{v} := \varphi'_H(\bar{x}; \bar{u}) \in Y$ completes the proof.
\end{proof}

The following remark demonstrates that this semicompactness property holds in a generalized equation system framework.

\begin{definition}\label{defn: strong approximation_in_y}\cite[Definition 2.3]{dontchev1995implicit}
    Let $\varphi: X \times Y \to Z$ and $g: Y \to Z$ be mappings between Banach spaces. $g$ is said to \emph{strongly approximate} $\varphi$ in $y$ at the reference point $(\bar{x}, \bar{y})$ if, for each $\varepsilon > 0$, there exist a neighborhood $U$ of $\bar{x}$ and a neighborhood $V$ of $\bar{y}$ such that 
    \begin{equation*}
        \big\| [\varphi(x, y_1) - g(y_1)] - [\varphi(x, y_2) - g(y_2)] \big\| \le \varepsilon \|y_1 - y_2\|,
    \end{equation*}
    whenever $x \in U$ and $y_1, y_2 \in V$. If $\varphi$ is Fr\'echet differentiable with respect to $y$ at $(\bar x,\bar y)$, then $g$ can be taken as $g(y)=\varphi(\bar x,\bar y)+\nabla_y \varphi(\bar x,\bar y)(y-\bar y)$.   
\end{definition} 

\begin{remark}\label{rem:strongly_directional_inner_semicompactness_ift}
Consider the generalized equation $P(x, y) \in \Lambda$ at a reference solution $(\bar{x}, \bar{y})$, and let $\Gamma(x) := \{ y \in Y \mid P(x, y) \in \Lambda \}$ be its associated solution mapping. 

Following \cite[Theorem 2.4]{dontchev1995implicit}, assume the following conditions hold:
(i) $P$ is Lipschitz continuous in $x$ uniformly with respect to $y$ at $(\bar{x}, \bar{y})$, and $P(\cdot, \bar{y})$ is Hadamard directionally differentiable at $\bar{x}$;
(ii) there exists a function $g$ that strongly approximates $P$ in $y$ at $(\bar{x}, \bar{y})$ with $g(\bar{y}) = P(\bar{x}, \bar{y})$;
(iii) the set-valued mapping $z \mapsto \{y \in Y \mid g(y) - z \in \Lambda\}$ admits a Hadamard single-valued localization $\xi$ at $(0, \bar{y})$. 

Under these conditions \cite[Theorem 2.4, Remark 2.6]{dontchev1995implicit}, $\Gamma$ admits a single-valued localization $\phi$ around $\bar{x}$ for $\bar{y}$. Moreover, this implicit function $\phi$ inherits the Hadamard directional differentiability at $\bar{x}$, with its directional derivative explicitly given by $\phi'_H(\bar{x}; \bar{u}) = \xi'_H\big(0; -P'_H(\cdot, \bar{y})(\bar{x}; \bar{u})\big)$ for any direction $\bar{u} \in X$. 
Consequently, Proposition \ref{prop: single_valued_implies_semicompact} ensures that $\Gamma$ is strongly directionally inner semicompact at $(\bar{x}, \bar{y})$ in the direction $(\bar{u}, \phi'_H(\bar{x}; \bar{u}))$.
\end{remark}
 Consider the linearization cone defined by
    \begin{equation*}
        \mathbb{L}(x, y; u) := \{ v \mid DP(x, y)(u, v) \cap T(P(x, y); \Lambda) \neq \emptyset \}.
    \end{equation*}
    Associated with this, we define the critical cone at $(\bar x,\bar y)$ in the direction $\bar u$ as 
    \begin{equation*}
        C(\bar{x}, \bar{y}; \bar{u}) := \left\{ \bar{v} \in \mathbb{L}(\bar{x}, \bar{y}; \bar{u}) \mid f'_{-}((\bar{x}, \bar{y}); (\bar{u}, \bar{v})) \leq V'_{+}(\bar{x}; \bar{u}), \; V'_{-}(\bar{x}; \bar{u}) \leq f'_{+}((\bar{x}, \bar{y}); (\bar{u}, \bar{v})) \right\}.
    \end{equation*}
Define the multiplier set in a direction $(\bar u,\bar v)$ at $(\bar x,\bar y)$:  
\begin{equation*}
    \mathcal{M}_P(\bar{x}, \bar{y}; \Lambda;\bar{u}, \bar{v}) := \left\{ z^* \in N(P(\bar{x}, \bar{y}); \Lambda; d) \;\middle|\; d = P'_H(\bar{x}, \bar{y}; \bar{u}, \bar{v}) \in T(P(\bar{x}, \bar{y});\Lambda) \right\}. 
\end{equation*}
When $S$ is strongly directional inner semicompact, we obtain the following theorem. It improves upon the results in \cite[Theorems~3.1(iv) and 3.2(iv)]{Bai2023Directional}, as it holds under weaker conditions even in finite-dimensional settings. This is because directional inner calmness implies strongly directional inner semicompactness, as noted in Remark~\ref{rem:inner_calm}.     

\begin{theorem}\label{thm: directional subdifferential of V with sdic of s}
Let $\bar{y} \in S(\bar{x})$ be given. Assume that the solution mapping $S$ is strongly directional inner semicompact at $(\bar{x}, \bar{y})$ in the direction $(\bar u, \bar v)$.  
Furthermore, assume that $P$ is Hadamard differentiable at $(\bar{x}, \bar{y})$ in the direction $(\bar u, \bar v)$, weak* strictly Lipschitzian at $(\bar{x}, \bar y)$ and $\Psi(x, y) := \Lambda - P(x, y)$ is metrically subregular in the direction $(\bar u, \bar v)$ at $(\bar{x}, \bar{y})$. Then, $\bar v\in C(\bar{x}, \bar{y}; \bar{u})$ and the directional subdifferential of $V$ satisfies:
\begin{equation*}
\begin{split}
    \partial V(\bar{x}; \bar u) \subset 
    \{ x^*  \;\mid\;
        (x^*, 0) \in \partial f(\bar{x}, \bar{y}; \bar u, \bar v) + \partial \langle z^*, P \rangle(\bar{x}, \bar{y}; \bar u, \bar v),~z^* \in \mathcal{M}_P(\bar{x}, \bar{y}; \Lambda; \bar u, \bar v ) \};
\end{split}
\end{equation*}
\begin{equation*}
\begin{split}
    \partial^\infty V(\bar{x}; \bar u) \subset \{ x^* ~\mid ~ 
 (x^*, 0) \in &\partial \langle z^*, P \rangle(\bar{x}, \bar{y}; \bar u, \bar v),~ z^* \in \mathcal{M}_P(\bar{x}, \bar{y};\Lambda; \bar u, \bar{v})
    \}.
\end{split}
\end{equation*}
\end{theorem}

\begin{proof}
    Let $x^* \in \partial V(\bar{x}; \bar{u})$. By definition, there exist sequences $t_k \downarrow 0$, $u_k \to \bar{u}$, and $x_k := \bar{x} + t_k u_k$ along with scalars $\varepsilon_k \downarrow 0$ such that local subgradients $x_k^* \in \hat{\partial} V(x_k)$ satisfying $x_k^* \xrightarrow{w^*} x^*$. 
    By the strongly directional inner semicompactness of $S$ at $(\bar x,\bar y)$ in the direction $(\bar u,\bar v)$, there exists a sequence $y_k \in S(x_k)$ such that $\frac{y_k - \bar{y}}{t_k} \to \bar{v}$ as $k \to \infty$.  
    
    Since $f(x_k, y_k) = V(x_k)$, $(x_k, y_k)$ is a local minimizer of the following problem:
    \begin{equation*}
        \min_{x, y} \left\{ \phi_k(x, y) := f(x, y) - \langle x_k^*, x - x_k \rangle + \varepsilon_k \|x - x_k\| + \delta_\Lambda(P(x, y)) \right\}.
    \end{equation*}
    Applying the fuzzy sum rule for Fr\'echet subdifferentials in Asplund spaces \cite[Theorem~2.33(b)]{Mordukhovich2006variational} (with the parameter $\gamma=\epsilon := t_k^2$ therein), we can find $e_{1,k} \in \mathbb{B}_X$ and $e_{2,k} \in \mathbb{B}_Y$ such that, by defining the perturbed points $(q_{1,k}, q_{2,k}) := (x_k + t_k^2 e_{1,k}, y_k + t_k^2 e_{2,k})$, we obtain the inclusion:
    \begin{equation*}
        (0, 0) \in \hat{\partial} f(q_{1,k}, q_{2,k}) - (x_k^*, 0) + (\varepsilon_k + 2t_k^2) \mathbb{B}^*_X \times \{0\} + \hat{\partial}(\delta_\Lambda \circ P)(q_{1,k}, q_{2,k}).
    \end{equation*}
    Note that the difference quotients of the perturbed sequences yield the same limits; that is, $\frac{q_{1,k} - \bar{x}}{t_k} \to \bar{u}$ and $\frac{q_{2,k} - \bar{y}}{t_k} \to \bar{v}$. 
    Taking the limit as $k \to \infty$ and invoking the definition of the directional limiting subdifferential along the direction $(\bar{u}, \bar{v})$, we arrive at: 
    \begin{equation*}
        (x^*, 0) \in \partial f(\bar{x}, \bar{y}; \bar{u}, \bar{v}) + \partial(\delta_\Lambda \circ P)(\bar{x}, \bar{y}; \bar{u}, \bar{v}).
    \end{equation*}
    
    Finally, since $P$ is Hadamard directionally differentiable and the metric subregularity holds, applying the chain rule (Theorem \ref{thm: chain rule}) yields: for $d := P'_H((\bar{x}, \bar{y}); (\bar{u}, \bar{v})) \in T(P(\bar{x}, \bar{y});\Lambda)$, we have 
    \begin{equation*}
        (x^*, 0) \in \partial f((\bar{x}, \bar{y}); (\bar{u}, \bar{v})) + \partial \langle z^*, P \rangle((\bar{x}, \bar{y}); (\bar{u}, \bar{v})), \quad \text{for some } z^* \in N(P(\bar{x}, \bar{y}); \Lambda; d).
    \end{equation*}
    Furthermore, following the arguments in \cite[Theorem 3.1(i), Case I]{bai2022directional}, we deduce that $\bar{v} \in C(\bar{x}, \bar{y}; \bar{u})$. This completes the proof for the upper estimate of $\partial V(\bar{x}; \bar{u})$. 
    
    The upper estimate for the directional singular subdifferential $\partial^\infty V(\bar{x}; \bar{u})$ follows analogously. By employing an auxiliary sequence $\ell_k \downarrow 0$ such that $\ell_k x_k^* \xrightarrow{w^*} x^*$ in the preceding arguments, the proof is identical; hence, the redundant details are omitted.
\end{proof}

Consider the lower-level problem:
\begin{equation}\label{eq:example_of_sdic}
    \min_y f(x, y) \quad \text{subject to} \quad P(x, y) \in \Lambda,
\end{equation}
where the constraint space $Z = \mathbb{R}^m$, $\Lambda = \mathbb{R}^m_-$, and $f(x, \cdot),P(x, \cdot)$ are convex.  
Assuming metric subregularity of $y\mapsto \Lambda-P(x,y)$, the solution mapping $S(x)$ is characterized by the Karush--Kuhn--Tucker (KKT) conditions. By introducing the Lagrange multiplier $\lambda \in \mathbb{R}^m$ and defining the primal-dual variable $w := (y, \lambda) \in Y \times \mathbb{R}^m$, the KKT system can be written as:
\begin{equation*}
    \nabla_y f(x, y) + \nabla_y P(x, y)^* \lambda = 0, \quad \lambda \in N(P(x, y);\Lambda).
\end{equation*}

For $\Lambda = \mathbb{R}^m_-$, its polar cone is $\mathbb{R}^m_+$, yielding the equivalence that $\lambda \in N_\Lambda(P(x,y)) \iff P(x,y) \in N_{\mathbb{R}^m_+}(\lambda)$. Therefore, by defining the base mapping $\mathcal{F}: X \times (Y \times \mathbb{R}^m) \to Y \times \mathbb{R}^m$ and the underlying convex set $K$ as follows:
\begin{equation*}
    \mathcal{F}(x, w) := 
    \begin{pmatrix} 
        \nabla_y f(x, y) + \nabla_y P(x, y)^* \lambda \\ 
        -P(x, y) 
    \end{pmatrix}, \quad 
    K := Y \times \mathbb{R}^m_+,
\end{equation*}
the primal-dual solution mapping is given by $S_{KKT}(x) := \{ w \in K \mid 0 \in \mathcal{F}(x, w) + N_K(w) \}$. 

The following corollary shows that, under metric regularity of the KKT system and suitable differentiability assumptions, the required directional semicompactness condition is satisfied. 

\begin{corollary}\label{cor: directional subdifferential kkt polyhedral}
    In the context of \eqref{eq:example_of_sdic}, let
    $\bar{y}\in S(\bar{x})$, let $\bar{u}\in X$ be a given direction,
    and let $\bar{\lambda}$ be a KKT multiplier at 
    $(\bar{x},\bar{y})$. Set $\bar{w}:=(\bar{y},\bar{\lambda})$. 
    Assume that the following conditions hold:
    \begin{enumerate}
        \item[(i)] $f$ and $P$ are continuously differentiable
        near $(\bar{x},\bar{y})$, with $\nabla_y f$ and
        $\nabla_y P$ being locally Lipschitz continuous and Hadamard
        directionally differentiable at $(\bar{x},\bar{y})$.

        \item[(ii)] The partial mapping $w\mapsto\mathcal{F}(\bar{x},w)+N_K(w)$ is metrically regular at $\bar{w}$ for $0$.
    \end{enumerate}
    Then the solution mapping $S$ admits a Hadamard single-valued localization $\phi$ at $(\bar{x},\bar{y})$ in the direction $\bar{u}$. Set $\bar{v}:=\phi'_H(\bar{x};\bar{u})$. 
    Assume in addition that: 
    \begin{enumerate}
        \item[(iii)] The mapping $P$ is weak* strictly Lipschitzian at
        $(\bar{x},\bar{y})$, and $\Psi(x,y):=\Lambda-P(x,y)$ 
        is metrically subregular at
        $\bigl((\bar{x},\bar{y}),0\bigr)$ in the direction
        $(\bar{u},\bar{v})$.
    \end{enumerate}
    Then $\bar{v}\in C(\bar{x},\bar{y};\bar{u})$, and the directional
    subdifferentials of the value function $V$ satisfy
    \begin{equation*}
    \begin{split}
        \partial V(\bar{x};\bar{u})
        \subset
        \Big\{x^*\;\Big|\;
            (x^*,0)\in
            \partial f(\bar{x},\bar{y};\bar{u},\bar{v})
            +\partial\langle z^*,P\rangle
            (\bar{x},\bar{y};\bar{u},\bar{v}),\;
            z^*\in
            \mathcal{M}_P(\bar{x},\bar{y};
            \Lambda;\bar{u},\bar{v})
        \Big\},
    \end{split}
    \end{equation*}
    \begin{equation*}
        \partial^\infty V(\bar{x};\bar{u})
        \subset
        \Big\{x^*\;\Big|\;
            (x^*,0)\in
            \partial\langle z^*,P\rangle
            (\bar{x},\bar{y};\bar{u},\bar{v}),\
            z^*\in
            \mathcal{M}_P(\bar{x},\bar{y};
            \Lambda;\bar{u},\bar{v})
        \Big\}.
    \end{equation*}
\end{corollary}

\begin{proof}
    Condition~(i) ensures that the base mapping $\mathcal{F}$ is locally Lipschitz continuous and Hadamard directionally differentiable at the reference point. By condition~(ii), the unperturbed KKT generalized equation is metrically regular at $\bar{w}$ for $0$. Due to the structural properties of KKT optimality systems \cite[Sections~4H and~4I.2]{dontchev2009implicit}, metric regularity in this setting implies strong metric regularity. Consequently, the
    mapping 
    \[
        w\mapsto\mathcal{F}(\bar{x},w)+N_K(w)
    \]
    admits a Lipschitz continuous single-valued inverse localization around $(\bar{w},0)$.

    Applying \cite[Theorem~2.4, Remark~2.6]{dontchev1995implicit} yields a locally Lipschitz continuous and Hadamard directionally differentiable
    localization $\Phi(x)=\bigl(\phi(x),\lambda(x)\bigr)$ 
    of the KKT solution mapping around $(\bar{x},\bar{w})$. More precisely, there exist neighborhoods $U_{\bar{x}}$ of $\bar{x}$ and $U_{\bar{w}}$ of $\bar{w}$ such that
    \[
        S_{KKT}(x)\cap U_{\bar{w}}=\{\Phi(x)\},
        \qquad x\in U_{\bar{x}},
    \]
    and $\Phi(\bar{x})=\bar{w}$.

Since the primal component $\phi$ is the primal component of the locally Lipschitz continuous and Hadamard directionally differentiable mapping $\Phi$, it is itself locally Lipschitz continuous and Hadamard directionally differentiable at $\bar{x}$. Thus $S$ admits a Hadamard single-valued localization at $(\bar{x},\bar{y})$ in the direction $\bar{u}$. By Proposition~\ref{prop: single_valued_implies_semicompact}, $S$ is strongly directionally inner semicompact at $(\bar{x},\bar{y})$ in the direction $(\bar{u},\bar{v})$, where $\bar{v}=\phi'_H(\bar{x};\bar{u})$. 

    Condition~(i) gives the required Hadamard directional differentiability
    of $P$, while condition~(iii) gives the weak* strict Lipschitzian
    property of $P$ and the directional metric subregularity of the
    original constraint mapping $\Psi=\Lambda-P$. Therefore all the
    assumptions of
    Theorem~\ref{thm: directional subdifferential of V with sdic of s}
    hold for the original problem \eqref{eq:example_of_sdic}. Applying that
    theorem gives $\bar{v}\in C(\bar{x},\bar{y};\bar{u})$ and the asserted
    upper estimates.
\end{proof}

We next record a verifiable sufficient condition for the metric regularity assumption in Corollary~\ref{cor: directional subdifferential kkt polyhedral}(ii). A related condition was applied in \cite[Theorem~5.1]{ye2000constraint}.  
In infinite dimensions, coderivative characterizations of metric regularity require additional sequential normal compactness, so we specialize to the finite-dimensional setting, where the Mordukhovich criterion applies directly.   

\begin{proposition}[NNAMCQ $\Rightarrow$ metric regularity]\label{prop: NNAMCQ metric regularity KKT}
Let $Y=\mathbb{R}^n$, $Z=\mathbb{R}^m$, $\Lambda=\mathbb{R}^m_-$, and $K:=\mathbb{R}^n\times\mathbb{R}^m_+$. 
Fix $\bar{x}\in X$ and a KKT point $\bar{w}=(\bar{y},\bar{\lambda})\in S_{KKT}(\bar{x})$, and define
\[
H(w)\ :=\ \mathcal{F}(\bar{x}, w) + N_K(w).
\]
Assume that $\mathcal{F}(\bar{x},\cdot)$ is continuously differentiable near $\bar{w}$. 
We say that the \emph{NNAMCQ for the KKT system} holds at $(\bar{w},0)$ if
\begin{equation}\label{eq:NNAMCQ_KKT}
    0\in \nabla_w\mathcal{F}(\bar{x},\bar{w})^*\xi + D^*N_K\bigl(\bar{w},-\mathcal{F}(\bar{x},\bar{w})\bigr)(\xi)
    \quad\Longrightarrow\quad
    \xi=0.
\end{equation}
Then $H$ is metrically regular at $\bar{w}$ for $0$. 
\end{proposition}

\begin{proof}
Since $K$ is polyhedral, $\operatorname{gph} N_K$ is closed and $N_K$ is a polyhedral multifunction. 
As $\mathcal{F}(\bar{x},\cdot)$ is continuously differentiable, $\operatorname{gph} H$ is locally closed at $(\bar{w},0)$, and the coderivative sum rule for a smooth single-valued mapping and a set-valued mapping with closed graph yields \cite[Theorem~1.62]{Mordukhovich2006variational}, for every $\xi\in \mathbb{R}^n\times\mathbb{R}^m$,
\begin{equation}\label{eq:coderiv_sum_H}
    D^*H(\bar{w},0)(\xi)
    =
    \nabla_w\mathcal{F}(\bar{x},\bar{w})^*\xi + D^*N_K\bigl(\bar{w},-\mathcal{F}(\bar{x},\bar{w})\bigr)(\xi). 
\end{equation} 
By \eqref{eq:coderiv_sum_H}, \eqref{eq:NNAMCQ_KKT} is exactly $D^*H(\bar{w},0)^{-1}(0)=\{0\}$. 
Applying the Mordukhovich criterion for metric regularity \cite[Corollary 4.3]{mordukhovich1993complete} to $H$ at $(\bar{w},0)$ completes the proof. 
\end{proof}

While the strong directional inner semicompactness of the mapping $S$ provides a theoretically powerful framework, it imposes a rather stringent requirement that can be overly restrictive in many practical settings. We next introduce a weaker condition, termed \textit{directional $V$-flatness}.  

\begin{definition}[Directional $V$-Flatness]\label{defn: gamma V flatness}
$f$ is said to be \emph{directionally $V$-flat} relative to $\Gamma$ at $\bar{x}$ in a direction $\bar{u} \in X$  
if for any sequences $t_k \downarrow 0$ and $u_k \to \bar{u}$, there exist $\bar v\in Y, \bar y\in S(\bar x;\bar u)$, and $y_k \in \Gamma(\bar{x} + t_k u_k)$ satisfying $\frac{y_k - \bar{y}}{t_k} \to \bar{v}$ and 
\begin{equation}\label{eq: V flatness}
    \liminf_{k \to \infty} \frac{f(\bar{x} + t_k u_k, y_k) - V(\bar{x} + t_k u_k)}{t_k} = 0.
\end{equation}
\end{definition}

This property is instrumental when employing variational principles to derive directional optimality conditions. In such frameworks, one typically constructs a sequence of near-minimizers $\{y_k\}$. Directional $V$-flatness ensures that the discrepancy between such a sequence $\{y_k\}$ and a corresponding sequence of minimizers $\{z_k\}$ vanishes at a rate faster than $t_k$. 

Since $\bar{y} \in S(\bar{x})$, this condition is weaker than the strong directional inner semicompactness of $S$ at $(\bar{x}, \bar{y})$. Indeed, the latter would allow one to select $y_k \in S(x_k)$ directly, in which case $f(x_k, y_k) - V(x_k)$ vanishes identically. 
Next, we give a proposition implying directional $V$-flatness. 

\begin{proposition}\label{prop: first_order_V_flatness_verification}
Suppose $\bar{y} \in S(\bar{x})$ and the following conditions hold:
\begin{enumerate}
    \item[\rm (i)] The feasible mapping $\Gamma$ is strongly directionally inner semicompact at $(\bar{x}, \bar{y})$ in the direction $(\bar{u}, \bar{v})$ for some $\bar{v} \in Y$. 
    
    \item[\rm (ii)] $f$ is locally Lipschitz near $(\bar{x}, \bar{y})$ and is Hadamard directionally differentiable at $(\bar{x}, \bar{y})$ in the direction $(\bar{u}, \bar{v})$, meaning the following limit exists:
    \begin{equation*}
        f'_H(\bar{x}, \bar{y}; \bar{u}, \bar{v}) := \lim_{t \downarrow 0, \, u \to \bar{u}, \, v \to \bar{v}} \frac{f(\bar{x} + t u, \bar{y} + t v) - f(\bar{x}, \bar{y})}{t}.
    \end{equation*}

    \item[\rm (iii)] The value function $V$ satisfies the lower directional bound at $\bar{x}$ in the direction $\bar{u}$ with respect to $\bar{y}$ and $\bar{v}$, namely:
    \begin{equation*}
        \liminf_{t \downarrow 0, \, u \to \bar{u}} \frac{V(\bar{x} + t u) - V(\bar{x})}{t} \ge f'_H(\bar{x}, \bar{y}; \bar{u}, \bar{v}).
    \end{equation*}
\end{enumerate}
Then directional $V$-flatness holds at $\bar{x}$ in the direction $\bar{u}$ for the pair $(f, \Gamma)$.
\end{proposition}

\begin{proof}
Let $t_k \downarrow 0$ and $u_k \to \bar{u}$ be given. Denote $x_k := \bar{x} + t_k u_k$. 
By condition (i), the strong directional inner semicompactness of $\Gamma$ at $(\bar{x}, \bar{y})$ in the direction $(\bar{u}, \bar{v})$ guarantees the existence of a sequence $y_k \in \Gamma(x_k)$ such that 
\begin{equation*}
    v_k := \frac{y_k - \bar{y}}{t_k} \to \bar{v} \quad \text{as} \quad k \to \infty.
\end{equation*}
Since $y_k \in \Gamma(x_k)$, it follows that $V(x_k) \le f(x_k, y_k)$ for all $k$. Consequently, the quotient is non-negative:
\begin{equation}\label{eq:prop_non_negative}
    \frac{f(x_k, y_k) - V(x_k)}{t_k} \ge 0, \quad \forall k.
\end{equation}

To prove the limit converges to zero, we evaluate the upper limit. Since $\bar{y} \in S(\bar{x})$, we have $V(\bar{x}) = f(\bar{x}, \bar{y})$. We decompose the quotient as follows:
\begin{align}\label{eq:prop_limsup_decomp}
    \limsup_{k \to \infty} \frac{f(x_k, y_k) - V(x_k)}{t_k} 
    &= \limsup_{k \to \infty} \left[ \frac{f(x_k, y_k) - f(\bar{x}, \bar{y})}{t_k} - \frac{V(x_k) - V(\bar{x})}{t_k} \right] \nonumber \\
    &\le \limsup_{k \to \infty} \frac{f(\bar{x} + t_k u_k, \bar{y} + t_k v_k) - f(\bar{x}, \bar{y})}{t_k} - \liminf_{k \to \infty} \frac{V(x_k) - V(\bar{x})}{t_k}.
\end{align}

By condition (ii), the limit of the first term exists and equals the directional derivative:
\begin{equation*}
    \lim_{k \to \infty} \frac{f(\bar{x} + t_k u_k, \bar{y} + t_k v_k) - f(\bar{x}, \bar{y})}{t_k} = f'_H(\bar{x}, \bar{y}; \bar{u}, \bar{v}).
\end{equation*}

Applying condition (iii) to the second term yields:
\begin{equation*}
    \liminf_{k \to \infty} \frac{V(x_k) - V(\bar{x})}{t_k} \ge f'_H(\bar{x}, \bar{y}; \bar{u}, \bar{v}).
\end{equation*}

Substituting both bounds back into \eqref{eq:prop_limsup_decomp}, we obtain:
\begin{equation*}
    \limsup_{k \to \infty} \frac{f(x_k, y_k) - V(x_k)}{t_k} \le f'_H(\bar{x}, \bar{y}; \bar{u}, \bar{v}) - f'_H(\bar{x}, \bar{y}; \bar{u}, \bar{v}) = 0.
\end{equation*}
Combining this with \eqref{eq:prop_non_negative}, we conclude that the limit exists and equals zero, confirming that directional $V$-flatness is satisfied.
\end{proof}

\begin{remark}\label{rem:verification_cond_iii}
Condition (iii) requires the lower Dini directional derivative of the value function to be bounded below by the directional derivative of the objective function. Although directly evaluating the directional derivative of $V(x)$ is challenging, this lower bound can be verified using sensitivity analysis for parametric optimization.

Depending on the problem structure and regularity conditions, Condition (iii) can be verified by ensuring $f'_H(\bar{x}, \bar{y}; \bar{u}, \bar{v})$ is bounded above by expressions involving the Lagrangian $L(x,y,\lambda)$ and either the limiting multiplier set $\Lambda(x,y)$ or the Clarke multiplier set $\Lambda^c(x,y)$: 

\begin{itemize}
    \item For parametric nonlinear programs, classical results by Gauvin and Dubeau \cite[Corollary 4.3]{gauvin1982differential} provide a lower bound over the full solution set $S(\bar{x})$. Under the uniform compactness of $S(x)$ near $\bar{x}$ and the Mangasarian-Fromovitz constraint qualification at each $y \in S(\bar{x})$,  
    \begin{equation*}
        \liminf_{t \downarrow 0, \, u \to \bar{u}} \frac{V(\bar{x} + t u) - V(\bar{x})}{t} \ge \inf_{y \in S(\bar{x})} \min_{\lambda \in \Lambda(\bar{x}, y)} \nabla_x L(\bar{x}, y, \lambda)\bar{u}.
    \end{equation*}
    
    \item Recent developments \cite[Section 4]{bai2024directional} tighten these bounds by optimizing over the directional solution set $S(\bar{x}; \bar{u})$. For general closed constraint sets and under directional Robinson stability, the lower bound is expressed using the Clarke multiplier set $\Lambda^c(\bar{x}, y)$:
    \begin{equation*}
        \liminf_{t \downarrow 0, \, u \to \bar{u}} \frac{V(\bar{x} + t u) - V(\bar{x})}{t} \ge \min_{y \in S(\bar{x}; \bar{u}) \cap \Omega_{\bar{u}}} \min_{\lambda \in \Lambda^c(\bar{x}, y)} \nabla_x L(\bar{x}, y, \lambda)\bar{u}.
    \end{equation*}

    \item If the constraint set is closed and convex, the Clarke multiplier set can be replaced by the sharper limiting multiplier set $\Lambda(\bar{x}, y)$ \cite{bai2024directional}:
    \begin{equation*}
        \liminf_{t \downarrow 0, \, u \to \bar{u}} \frac{V(\bar{x} + t u) - V(\bar{x})}{t} \ge \min_{y \in S(\bar{x}; \bar{u}) \cap \Omega_{\bar{u}}} \min_{\lambda \in \Lambda(\bar{x}, y)} \nabla_x L(\bar{x}, y, \lambda)\bar{u}.
    \end{equation*}

\end{itemize}

Furthermore, analogous lower bounds can be explicitly established for specifically structured problems, such as parametric quadratic programs \cite[Theorem 4.1]{Tam2001directional}, \cite[Theorem 4]{Dong2024directional}.
\end{remark}

\begin{example}\label{ex: V_flatness_verification}

Let $x \in \mathbb{R}$ and $y = (y_1, y_2) \in \mathbb{R}^2$. Consider the following parametric quadratic programming problem:
\begin{equation*}
    \min_{y \in \mathbb{R}^2} \quad f(x, y) = y_1 + \frac{1}{2}y_2^2 \quad \text{s.t.} \quad g(x, y) = \frac{1}{2}y_1^2 + \frac{1}{2}y_2^2 - \frac{1}{2} + x \le 0.
\end{equation*}

Set $\bar{x} = 0$ and $\bar{u} = 1$.  
When $\bar{x} = 0$, the feasible set is $y_1^2 + y_2^2 \le 1$. Minimizing $y_1 + \frac{1}{2}y_2^2$ is uniquely achieved on the boundary, yielding the optimal solution $\bar{y} = (-1, 0)$ and $V(0) = -1$.
The constraint is active at $\bar{y}$. The Lagrangian function is defined as:
\begin{equation*}
    L(x, y, \lambda) = y_1 + \frac{1}{2}y_2^2 + \lambda \left(\frac{1}{2}y_1^2 + \frac{1}{2}y_2^2 - \frac{1}{2} + x\right).
\end{equation*}
From the first-order KKT condition $\nabla_y L(\bar{x}, \bar{y}, \bar{\lambda}) = (1 - \bar{\lambda}, 0) = (0, 0)$, we obtain the unique Lagrange multiplier $\bar{\lambda} = 1$, meaning $\Lambda(0, \bar{y}) = \{1\}$.

For a small positive perturbation $x_t = t > 0$, the feasible set updates to $y_1^2 + y_2^2 \le 1 - 2t$. The exact optimal solution becomes $y(t) = (-\sqrt{1-2t}, 0)$, and the optimal value function is $V(t) = -\sqrt{1-2t}$. We now verify the three conditions of the proposition:

\begin{enumerate}
    \item[\rm (i)] 
We choose $y_t := y(t) = \bigl(-\sqrt{1-2t},0\bigr) \in \Gamma(x_t)$, which defines a Hadamard directionally differentiable single-valued localization. Hence, by Proposition~\ref{prop: single_valued_implies_semicompact}, the strong directional inner semicompactness condition is satisfied.

    \item[\rm (ii)] 
    $f(x, y) = y_1 + \frac{1}{2}y_2^2$ is smooth ($C^\infty$). Its joint directional derivative is given by:
    \begin{equation*}
        f'_H(\bar{x}, \bar{y}; \bar{u}, \bar{v}) = \nabla_x f \cdot \bar{u} + \nabla_y f \cdot \bar{v} = 0 \cdot 1 + (1, 0) \cdot (1, 0) = 1.
    \end{equation*}
    Hence, Condition (ii) holds with a directional derivative value of $1$.

    \item[\rm (iii)] The assumptions of \cite[Corollary~4.3]{gauvin1982differential} are satisfied at $(0,-1,0)$. Therefore,
    \begin{equation*}
            V'(0; 1) \geq \min_{y \in S(0)} \min_{\lambda \in \Lambda(0, -1,0)} \nabla_x L(0, -1, 0, \lambda) \cdot \bar{u}=1, 
        \end{equation*}
        with $\Lambda(0, -1,0) = \{1\}$. Thus $V'(0; 1)= f'_H(\bar x,\bar y;\bar u,\bar v)$. 
\end{enumerate}
\end{example}

Theorem~\ref{thm: directional subdifferential of V} requires weaker assumptions than \cite[Theorem~5.10]{long2017calculus}, which assumes strong directional inner semicompactness.   

\begin{theorem}\label{thm: directional subdifferential of V}
Let $\bar{x} \in X$ and a direction $\bar u \in X$ be given. Suppose $(f, \Gamma)$ satisfies directional $V$-flatness and $V$ satisfies directional restricted inf-compactness with $\Omega_{\bar u}$ at $\bar{x}$ in the direction $\bar u$. Assume that, for each $\bar{y} \in S(\bar{x};\bar u) \cap \Omega_{\bar u}$ and each direction $\bar v(\bar{y}) \in Y$ produced by this directional $V$-flatness, $P$ is Hadamard differentiable at $(\bar{x}, \bar{y})$ in the direction $(\bar u, \bar v(\bar{y}))$ and weak* strictly Lipschitzian at $(\bar{x}, \bar y)$, and the set-valued mapping $\Psi(x, y) := \Lambda - P(x, y)$ is metrically subregular at $((\bar{x}, \bar{y}), 0)$ in the direction $(\bar u, \bar v(\bar{y}))$. Then 
\begin{equation*}
    \partial V(\bar{x}; \bar u) \subset  \bigcup\limits_{\bar y\in S(\bar x; \bar u)\bigcap \Omega_{\bar u}} \left\{ x^* \;\middle|\;
    \begin{aligned}
        &(x^*, 0) \in \partial f(\bar{x}, \bar{y}; \bar u, \bar v(\bar{y})) + \partial \langle z^*, P \rangle(\bar{x}, \bar{y}; \bar u, \bar v(\bar{y})), \\
        &z^* \in \mathcal{M}_P(\bar{x}, \bar{y}; \Lambda; \bar{u}, \bar{v}(\bar y))
    \end{aligned}
    \right\}, 
\end{equation*}
\begin{equation*}
\begin{split}
    \partial^\infty V(\bar{x}; \bar u) \subset \bigcup\limits_{\bar y\in S(\bar x;\bar u)\bigcap \Omega_{\bar u}}\{ x^* ~\mid ~ 
 (x^*, 0) \in &\partial \langle z^*, P \rangle((\bar{x}, \bar{y}); (\bar u, \bar v(\bar{y}))),~ z^* \in \mathcal{M}_P(\bar{x}, \bar{y};\Lambda; \bar u, \bar{v}(\bar y))
    \}.
\end{split}
\end{equation*}
\end{theorem}

\begin{proof}
Let $x^* \in \partial V(\bar{x}; \bar u)$. By the definition of the directional subdifferential in Asplund spaces (cf. \cite[Theorem 2.34]{Mordukhovich2006variational}), there exist sequences $t_k \downarrow 0$, $u_k \to \bar u$, and $x_k = \bar{x} + t_k u_k$ such that $V(x_k) < V(\bar{x}) + \varepsilon_k$ with $\varepsilon_k \downarrow 0$, and $x_k^* \in \hat{\partial} V(x_k)$ such that $x_k^* \xrightarrow{w^*} x^*$. Because of the directional restricted inf-compactness of $V$, there exists $z_k \in S(x_k) \cap \Omega_{\bar u}, \bar y\in S(\bar{x};\bar u)\cap \Omega_{\bar u}$ such that $z_k \to \bar{y}$. 

Since $f(x_k, z_k) = V(x_k)$, the pair $(x_k, z_k)$ is a local minimizer of the problem:
\begin{equation*}
    \min_{x, z} \left\{ f(x, z) - \langle x_k^*, x - x_k \rangle + \varepsilon_k \|x - x_k\| + \delta_\Lambda(P(x, z)) \right\}.
\end{equation*}
By directional $V$-flatness, we can find (after passing to a subsequence) $y_k \in \Gamma(x_k)$ and $v(\bar y)\in Y$ such that $\frac{y_k - \bar{y}}{t_k} \to \bar v(\bar y)$. Define 
$\eta_k := |f(x_k, y_k) - f(x_k, z_k)| = |f(x_k, y_k) - V(x_k)|$. 
Then $(x_k, y_k)$ is an $\eta_k$-minimizer of the following functional:
\begin{equation*}
    \phi_k(x, y) := f(x, y) - \langle x_k^*, x - x_k \rangle + \varepsilon_k \|x - x_k\| + \delta_\Lambda(P(x, y)).
\end{equation*}
Directional $V$-flatness of $(f, \Gamma)$ ensures that $\frac{\eta_k}{t_k} \to 0$. Setting $\lambda_k := \sqrt{t_k \eta_k}$, we have $\frac{\eta_k}{\lambda_k} \to 0$ and $\frac{\lambda_k}{t_k} \to 0$. Applying the subdifferential variational principle in Asplund spaces \cite[Theorem~2.28]{Mordukhovich2006variational}, there exist $(h_{1,k}, h_{2,k}) \in X \times Y$ and $(h_{1,k}^*, h_{2,k}^*) \in \hat{\partial} \phi_k(x_k + h_{1,k}, y_k + h_{2,k})$ such that:
\begin{itemize}
    \item[(1)] $(x_k + h_{1,k}, y_k + h_{2,k})$ is feasible, i.e., $P(x_k + h_{1,k}, y_k + h_{2,k}) \in \Lambda$;
    \item[(2)] $\|(h_{1,k}, h_{2,k})\| \leq \lambda_k$, which implies $\frac{(x_k + h_{1,k}) - \bar{x}}{t_k} \to \bar u$ and $\frac{(y_k + h_{2,k}) - \bar{y}}{t_k} \to \bar v(\bar y)$;
    \item[(3)] $\|(h_{1,k}^*, h_{2,k}^*)\| \leq \frac{\eta_k}{\lambda_k} \to 0$.
\end{itemize}

From the inclusion $(h_{1,k}^*, h_{2,k}^*) \in \hat{\partial} \phi_k(x_k + h_{1,k}, y_k + h_{2,k})$, we apply the fuzzy sum rule for Fr\'echet subdifferentials in Asplund spaces \cite[Theorem~2.33(b)]{Mordukhovich2006variational} (setting $\gamma=\epsilon := \lambda_k^2$ therein). This guarantees the existence of $e_{1,k} \in B_X$ and $e_{2,k} \in B_Y$ such that, by defining $(q_{1,k}, q_{2,k}) := (x_k + h_{1,k} + \lambda_k^2 e_{1,k}, y_k + h_{2,k} + \lambda_k^2 e_{2,k})$, we have 
\begin{equation*}
    (h_{1,k}^*, h_{2,k}^*) \in \hat{\partial} f(q_{1,k}, q_{2,k}) - (x_k^*, 0) + (\varepsilon_k + 2\lambda_k^2) \mathbb{B}^*_X \times \{0\} + \hat{\partial}(\delta_\Lambda \circ P)(q_{1,k}, q_{2,k}).
\end{equation*}
Observe that the difference quotients still converge to the same limits, i.e., $\frac{q_{1,k} - \bar{x}}{t_k} \to \bar{u}$ and $\frac{q_{2,k} - \bar{y}}{t_k} \to \bar{v}(\bar{y})$. 
Passing to the limit as $k \to \infty$ and utilizing the definition of the directional limiting subdifferential along the direction $(\bar{u}, \bar{v}(\bar{y}))$, we obtain: 
\begin{equation*}
    (x^*, 0) \in \partial f(\bar{x}, \bar{y}; \bar{u}, \bar{v}(\bar{y})) + \partial(\delta_\Lambda \circ P)(\bar{x}, \bar{y}; \bar{u}, \bar{v}(\bar{y})).
\end{equation*}

Finally, since $P$ is Hadamard differentiable and the metric subregularity holds, the chain rule (Theorem \ref{thm: chain rule}) yields: for $d := P'_H((\bar{x}, \bar{y}); (\bar u, \bar v(\bar{y}))) \in T(P(\bar{x}, \bar{y});\Lambda)$, we have 
\begin{equation*}
    (x^*, 0) \in \partial f((\bar{x}, \bar{y}); (\bar u, \bar v(\bar y))) + \partial \langle z^*, P \rangle((\bar{x}, \bar{y}); (\bar u, \bar v(\bar y))), \quad z^* \in N(P(\bar{x}, \bar{y}); \Lambda; d).
\end{equation*}
This completes the proof of the upper estimate for \(\partial V(\bar{x}; \bar{u})\). The upper estimate for \(\partial^\infty V(\bar{x}; \bar{u})\) follows analogously by deploying an auxiliary sequence \(\ell_k \downarrow 0\) such that \(\ell_k x_k^* \xrightarrow{w^*} x^*\) in the preceding arguments, and the redundant details are thus omitted. 
\end{proof}

\begin{remark}\label{rem: critical_cone}
    In the finite-dimensional framework established in \cite[Theorem 3.1]{Bai2023Directional}, the direction $\bar{v}(\bar{y})$ is shown to belong to the critical cone $C(\bar{x}, \bar{y}; \bar{u})$. However, in Theorem~\ref{thm: directional subdifferential of V}, we can only guarantee that $\bar{v}(\bar{y}) \in \mathbb{L}(\bar{x}, \bar{y}; \bar{u})$. This discrepancy arises because, in the infinite-dimensional setting, we cannot guarantee the existence of a sequence $y_k \in S(x_k)$ satisfying $\frac{y_k - \bar{y}}{t_k} \to \bar{v}(\bar{y})$. Consequently, we are unable to establish the following sequence of estimates:
    \begin{equation*}
        f'_{-}((\bar{x}, \bar{y}); (\bar{u}, \bar{v})) \leq \lim_{k \to \infty} \frac{f(\bar{x} + t_k u_k, y_k) - f(\bar{x}, \bar{y})}{t_k} = \lim_{k \to \infty} \frac{V(\bar{x} + t_k u_k) - V(\bar{x})}{t_k} \leq V'_{+}(\bar{x}; \bar{u}).
    \end{equation*} 
\end{remark}

By contrast, the absence of strong compactness in Asplund spaces precludes the use of classical sequential techniques in Theorem~\ref{thm: directional subdifferential of V}, making this critical cone restriction unattainable. 

Characterizations of $\partial V(\bar{x}; \bar u)$ and $\partial^\infty V(\bar{x}; \bar u)$ were investigated in \cite{long2017calculus} for general Banach spaces.  While their approach does not necessitate the Hadamard differentiability of $P$, the resulting expressions remain somewhat implicit, as they are formulated in terms of the coderivative $D_N^* \Gamma$ and involve auxiliary functions. 

In contrast, by leveraging the Hadamard differentiability of $P$ and the directional normal cone calculus established in Theorem \ref{thm: directional normal cones of preimage sets}, 
Theorem \ref{thm: directional subdifferential of V}, and Theorem \ref{thm: directional subdifferential of V with sdic of s} extend the results in \cite[Theorems 5.10, 5.11]{long2017calculus} by relaxing the assumptions on the solution mapping $S$ and providing a scalarization of coderivatives within the framework of Asplund spaces. 

Although this explicit characterization requires $P$ to be Hadamard differentiable, such a requirement remains sufficiently broad for practical applications; indeed, many standard works, such as \cite[Subsection 1.4.2]{hinze2008optimization}, impose the significantly more restrictive condition of strict Fr\'{e}chet differentiability. 

Furthermore, Theorem~\ref{thm: directional subdifferential of V} generalizes the finite-dimensional results of \cite[Theorems~3.1(iv) and~3.2(iv)]{Bai2023Directional} to the Asplund space setting. In finite dimensions, boundedness inherently implies compactness, a luxury not available in general Asplund spaces. By addressing the lack of automatic compactness through our refined variational analysis, our estimates provide a more general perspective on directional subdifferentials.

\vskip 6mm

\noindent

\bibliographystyle{plain}
\bibliography{sn-bibliography}

\end{document}